\documentclass[a4paper,reqno,11pt]{amsart}

\usepackage{graphicx} 
\usepackage[T1]{fontenc} 
\usepackage[utf8]{inputenc}
\usepackage[english]{babel} 
\usepackage{hyperref}
\usepackage{xcolor}
\usepackage{amsmath,amsfonts,amsthm,amssymb}
\usepackage{enumitem}
\usepackage{graphicx}

\numberwithin{equation}{section}

\theoremstyle{plain} 
\newtheorem{theorem}{Theorem}[section]
\newtheorem*{theorem*}{Theorem}

\newtheorem{proposition}[theorem]{Proposition}
\newtheorem*{proposition*}{Proposition}

\newtheorem{corollary}[theorem]{Corollary}
\newtheorem*{corollary*}{Corollary}

\newtheorem{lemma}[theorem]{Lemma}
\newtheorem*{lemma*}{Lemma}

\newtheorem*{conjecture*}{Conjecture}

\theoremstyle{definition} 
\newtheorem{definition}[theorem]{Definition}
\newtheorem*{definition*}{Definition}

\theoremstyle{remark} 
\newtheorem{remark}[theorem]{Remark}
\newtheorem*{remark*}{Remark}

\newtheorem*{example*}{Example}

\newcommand{\diff}{\mathrm{d}}

\title[]{Smoothness of martingale observables and generalized Feynman--Kac formulas}

\author[Karrila]{Alex Karrila}
\address{\AA bo Akademi Matematik; Henriksgatan 2, 20500 \AA bo, Finland}
\email{alex.karrila@abo.fi}

\author[Viitasaari]{Lauri Viitasaari}
\address{Aalto University School of Business, Department of Information and Service Management, PO Box 11110, 00076 Aalto, Finland}
\email{lauri.viitasaari@aalto.fi}

\begin{document}

	\keywords{Degenerate diffusions, Feynman-Kac formula, martingale observables, Schramm--Loewner evolutions}
	\subjclass[2020]{60H10, 60H30, 60J67, 60G44}
	\thanks{}
	\date{\today}
	\begin{abstract}
		%
		We prove that, under the H\"ormander criterion on an It\^{o} process, all its martingale observables are smooth. As a consequence, we also obtain a generalized Feynman--Kac formula providing smooth solutions to certain PDE boundary-value problems, while allowing for degenerate diffusions as well as boundary stopping (under very mild boundary regularity assumptions). We also highlight an application to a question posed on Schramm--Loewner evolutions, by making certain Girsanov transform martingales accessible via It\^{o} calculus.
	\end{abstract}

	\maketitle
	
	{\small
		\tableofcontents
	}

	\section{Introduction}
	The celebrated Feynman--Kac formulas represent solutions to parabolic second order partial differential equations (PDEs) as conditional expectations of functionals of It\^o processes, sparking advances in both probability and analysis.
	Classical applications include the derivation of the Black--Scholes model for option pricing~\cite{Karatzas-Shreve98} and Monte Carlo methods for numerically solving PDEs~\cite{MC-book, diffusion-MC}.
	While classical and stochastic analysis are often parallel, the latter is generally less evasive of \textit{degenerate diffusions}, for which the second-order part of the parabolic PDE --- in the notation of Section~\ref{sec:main}, the positive semi-definite, symmetric matrix $a(x)$ --- is not uniformly elliptic, i.e., its smallest eigenvalue attains the value zero or tends to zero on the domain boundary. The former type naturally arises, e.g., in particle kinetics~(\cite[Section~2.4]{Villaini}; see also~\cite{Doren}) and Conformal Field Theory and SLEs~(Section~\ref{subsec:SLE}), and the latter in, e.g., finance,~\cite{Heston, Karatzas} and mathematical biology~\cite{Pop1}. Especially in the degenerate case, generalized Feynman--Kac formulas also turn into a theoretical tool, as the comparatively complete theory of uniformly elliptic PDEs is not at hand, cf.~\cite{Pop2, WZ22}.
	
	This article studies the relation of parabolic PDEs, Feynman--Kac formulas and \textit{martingale observables} (introduced below), under the general H\"ormander criterion on the underlying It\^{o} process $X_t$. Emphasis is put on smoothness questions, as they play a crucial r\^{o}le in any further It\^{o} calculus computations (cf. Section~\ref{subsec:SLE}). We make no assumption towards ellipticity or uniform ellipticity of the operators (resp. diffusions), remaining thus in the domain where Feynman--Kac type formulas are particularly valuable and the existing theory is least extensive.

	To explain intuitions and caveats around these relations,
	in the simplest special case (cf. Definition~\ref{def:mgale obs}), a martingale observable refers to a function $f: \Lambda \times [0, T) \to \mathbb{R}$ in space $\Lambda \subset \mathbb{R}^n$ and time $t \in [0, T)$ such that, for a given It\^o diffusion $X_t$ on $\Lambda$, $f(X_t,t)$ is a local martingale.
	On the one hand, restricting to the class of twice continuously differentiable functions $f$, one can prove by direct It\^{o} calculus that $f$ is such a martingale observable if and only if it satisfies the PDE addressed in the Feynman--Kac formula. However, there is no reason \textit{a priori} that a martingale observable, or a generalized PDE solution (without ellipticity), should be twice continuously-differentiable. On the other hand, if $f$ is given by a Feynman--Kac type expectation $f(x, t )= \mathbb{E}[\psi(X_\tau, \tau) \; |\; X_t = x]$ of some bounded boundary values $\psi$ at the hitting time $\tau$ of $\partial ( \Lambda  \times [0, T) )$, then $f(X_t, t)$ is clearly a martingale. However, not every martingale observable has such an expression (the observable may not be a genuine martingale, 
	$X_\tau$ may not be defined, etc.), and even if it has, in $n > 1$ spatial dimensions, such an expectation is not necessarily a twice-differentiable function (nor a strong PDE solution).

	As our first main result, Theorem \ref{thm:mgale observables solve PDEs}, we prove that martingale observables are automatically weak solutions to the related PDEs. By basic PDE regularity theory, one then can deduce smoothness of the martingale observables (and hence equivalence to PDEs) provided that the order zero and driving term of the related PDE are smooth.\footnote{Smoothness of higher-order terms is already required by Hörmander's criterion.}
	We emphasize that, on top of not requiring ellipticity, we neither require the coefficient functions of the PDE to be bounded, typically leading to a finite life-time of the underlying It\^{o} process. In particular, this is the case in our two main applications concerning SLE curves, see Section \ref{subsec:SLE}. Roughly speaking, the SLE interpretations of our result are (i) that certain SLE convergence results for lattice models can be fairly simply extended to variant lattice models (cf.~\cite{FW}), and (ii) a piece in the construction of perturbed SLE models, bridging between martingale observables and an explicit expression relying on It\^{o} calculus of smooth functions~(cf. \cite{PW, HPW}). In both cases, the martingale observables are Girsanov's measure-transforming martingales between different SLE type models. Our result will also be used in an ongoing work~\cite{ongoing} for a conclusion of type~(ii); in this paper we exemplify such use with multiple SLEs, for which an alternative proof is already known~\cite{Dubedat, AHPY, FLPW}.

	
	Largely following from the first main result, Theorem \ref{thm:Feynman--Kac} provides a generalized Feynman--Kac formula, giving probabilistic representations for solutions to boundary value problems related to parabolic second order PDEs. Our result merely assumes H\"ormander's criterion and some very mild boundary regularity of $\Lambda$ (again, no (uniform) ellipticity), hence widening the knowledge related to Feynman-Kac formulas, to the best of our knowledge. As with the martingale observables, smooth strong solutions are guaranteed for smooth coefficient functions. As a special case, in Theorem \ref{thm:X-harmonic FK} we re-prove the solution of so-called $X$-harmonic problems, studied, e.g., in~\cite{Bony, Stroock-Var, Oksendal-SDE}. 
	
	Our proof is based on combining tools from PDE theory and stochastic calculus. The core new idea towards Theorem~\ref{thm:mgale observables solve PDEs} is to introduce a slowed-down process $\hat{X}$, a random time-change of $X$ that preserves the martingale observables of $X$ but is forced by the slowing-down effect to stay inside a bounded subset $\Theta \subset \Lambda$. This is a stochastic analogue of cutoff function arguments in analysis. Transforming calculations to the level of the slowed-down process then allows us to avoid boundary problems and to handle with the original process $X$ having a finite life-time. With the latter problem cured and with H\"{o}rmander's criterion assumed, $\hat{X}_t$ are well-known to have smooth particle densities that solve Kolmogorov's PDEs (recalled in Theorem~\ref{thm:existence of a smooth density}). These are dual PDEs to the weak PDE for martingale observables, which are then obtained after careful limit arguments to remove the cutoff functions.
	
	The rest of the article is organized as follows. In Section \ref{sec:main} we present and discuss our main results. The smoothness of martingale observables is discussed in Section \ref{subsec:smoothness} while the generalized Feynman-Kac formula and the case of $X$-harmonic problems is presented in Sections \ref{subsec:FK}--\ref{subsec:intro X-harmonic}. The application to SLEs is presented in Section \ref{subsec:SLE}. The proofs are presented in Section \ref{sec:proofs}: we begin with some preliminaries on known results in Section \ref{subsec:PDE background}, the slowed-down process is introduced in Section \ref{subsec:slowed-down}, and the proofs of our main results are presented in Sections \ref{subsec:main-proof1}-\ref{subsec:main-proof2}. Finally, proofs of some technical lemmas are postponed to Section \ref{subsec:technical stuff}. 
	
	\subsubsection*{\textbf{Acknowledgements}} We thank Benny Avelin, Eveliina Peltola and Paavo Salminen for insightful discussions. A.K. was supported by the Academy of Finland, grant number 339515.

	\section{Main results}
	\label{sec:main}
	Throughout this article, we fix the following setup and notations. Let $\Lambda \subset \mathbb{R}^n$ be an open set. We consider an $\mathbb{R}^n$-valued It\^{o} processes $X_t$ which, up to the hitting time to $\Lambda^c$, satisfies the time-homogeneous SDE
	\begin{align}
		\label{eq:main process of interest}
		\diff X^{i}_t = \sum_{j=1}^d \sigma_{i,j}(X_t) \diff B^{j}_t +  b_{i}(X_t) \diff t, \qquad 1 \leq i \leq n,
	\end{align}
	where $B_t  = (B^1_t,\ldots, B^d_t)$ is a $d$-dimensional Brownian motion independent of $X_0$ and $\sigma: \Lambda \to \mathbb{R}^{n \times d}$ and $b: \Lambda \to \mathbb{R}^n$ are componentwise \textit{smooth} functions. The related filtered probability space is $(\Omega, \mathcal{F}, (\mathcal{F}_t)_{t \geq 0}, \mathbb{P})$, where the filtration $(\mathcal{F}_t)_{t \geq 0}$ satisfies the usual assumptions, i.e., it is right-continuous and $\mathcal{F}_0$ contains all $\mathbb{P}$-null sets.
	
	Let us also define $a: \Lambda \to \mathbb{R}^{n \times n}$ as $a(x)=\sigma(x) \sigma(x)^T$, i.e., $a_{i, j}(x) = \sum_{q=1}^d \sigma_{i, q}(x)\sigma_{j,q}(x)$, and the differential operators $\mathcal{G}$ on twice differentiable functions $\Lambda \to \mathbb{R}$ via
	\begin{align}
		\label{eq:spatial generator}
		\mathcal{G} = \tfrac{1}{2} \sum_{i=1}^n \sum_{j=1}^n a_{i, j}(x) \partial_{ij} + \sum_{i = 1}^n  b_{i}(x) \partial_i,
	\end{align}
	i.e., $\mathcal{G}$ is the spatial-variable part of the generator of $X$. The dual operator $\mathcal{G}^*$ of $\mathcal{G}$ is defined by its action on twice-differentiable functions $\varphi: \Lambda \to \mathbb{R}$:
	\begin{align*}
		\big( \mathcal{G}^* \varphi \big) (x) = \tfrac{1}{2} \sum_{i=1}^n \sum_{j=1}^n \partial_{ij} \big( a_{i, j}(x) \varphi(x) ) - \sum_{i = 1}^n \partial_i \big( b_{i}(x) \varphi(x) \big).
	\end{align*}
	We also define the smooth vector fields
	\begin{align}
		\label{eq:Lie generators 1}
		\mathcal{U}_q &:= \sum_{i=1}^n \sigma_{i, q} (x) \partial_i, \qquad q=1, \ldots, d \qquad \text{and} \\
		\label{eq:Lie generators 2}
		\mathcal{U}_0 &:= \sum_{i= 1}^n b_i(x) \partial_i - \tfrac{1}{2} \sum_{q= 1}^d \sum_{i= 1}^n \big( \mathcal{U}_q \sigma_{i, q} (x) \big) \partial_i.
	\end{align}
	Note that whence $\mathcal{G} = \tfrac{1}{2} \sum_{q=1}^d \mathcal{U}^2_q + \mathcal{U}_0$, 
	and the above vector fields appear in (several variants of) H\"{o}rmander's theorem. In particular, we shall repeatedly refer to the condition (see Section~\ref{subsec:PDE background} for details):
	\begin{align}
		\nonumber
		& \text{Condition \eqref{eq:Hormander criterion}: The Lie algebra generated by } \{ \mathcal{U}_q \}_{q=1}^d \\
		\label{eq:Hormander criterion}
		\tag{H}
		&  \text{and } \{ [\mathcal{U}_q, \mathcal{U}_0] \}_{q=1}^d \text{ is of dimension $n$ at every $x \in \Lambda$.}
	\end{align}

	\subsection{Smoothness of martingale observables}
	\label{subsec:smoothness}
	Our first main result is that, under the H\"{o}rmander condition on the process $X$, \textit{martingale observables} are weak (often strong) solutions to the second order generator PDE.
	
	\begin{definition}
		\label{def:mgale obs}
		Fix $T>0$ and let $f: \Lambda \times [0,T) \to \mathbb{R}$, $ g: \Lambda  \to \mathbb{R}$ and $ h: \Lambda \times [0,T) \to \mathbb{R}$ be Borel measurable functions. We say that $f$ is a \textit{$(g, h)$-martingale observable for the process $X$} (for short, a \textit{martingale observable}) if for any launching time $t_0 \in [0,T)$ and point $X_{t_0} = x_0 \in \Lambda$
		\begin{align}
			\label{eq:def of mgale observable}
			M_t &: = \gamma_t f(X_t, t) + H_t, \qquad t  > t_0, \qquad \text{where}\\
			\nonumber
			\gamma_t = \gamma_{t_0, t } &:= \exp \left( \int_{s= t_0}^t g(X_s) \diff s \right) \quad \text{and} \quad H_t = H_{t_0, t} := \int_{s=t_0}^t \gamma_s h(X_s, s) \diff s ,
		\end{align}
		is a local martingale up to the exit time of $\Lambda \times [0, T)$ by the pair $(X_t, t)$.
	\end{definition}
	
	\begin{theorem}
		\label{thm:mgale observables solve PDEs}
		Fix $T>0$ and let $f: \Lambda \times [0,T) \to \mathbb{R}$ be locally bounded and Borel measurable, and $ g: \Lambda  \to \mathbb{R}$ and $ h: \Lambda \times [0,T) \to \mathbb{R}$ continuous.
		If $f$ is a $(g,h)$-martingale observable for the process $X$,
		and condition~\eqref{eq:Hormander criterion} is satisfied for that process, then $f$ is a weak solution to the PDE
		\begin{align}
			\label{eq:thm-weak-pde}
			\mathcal{G} f(x, t)  + \partial_t f(x, t) + g(x) f (x, t) + h(x, t) = 0, \quad (x, t) \in \Lambda \times (0, T),
		\end{align}
		i.e., for all smooth, compactly-supported $\varphi: \Lambda \times (0, T) \to \mathbb{R}$, we have
		\begin{align}
			\label{eq:weak main PDE}
			\int_{x \in \Lambda} \int_{t=0}^T  \Big( f(x, t) ( \mathcal{G}^*  - \partial_t + g(x)) \varphi (x, t) + h(x, t) \varphi (x, t) \Big) \diff t \diff^n x = 0.
		\end{align}
	\end{theorem}
	
	The classical question of regularity of a weak solution is of course particularly interesting in the context of It\^{o} calculus (for one application, see Section~\ref{subsec:SLE}).
	The above assumption on Lie algebras implies the H\"ormander condition for the operator $\mathcal{G} + \partial_t$ (see Remark~\ref{rem:different Lie algebras}). In particular, if the functions $g$ and $h$ are smooth, $f$ is a smooth strong solution to the PDE~\eqref{eq:thm-weak-pde}. This is what we mean by ``smoothness of martingale observables''. Conversely, straightforward It\^{o} calculus shows that for a smooth solution $f$ to the PDE~\eqref{eq:thm-weak-pde}, the process~\eqref{eq:def of mgale observable} is a martingale. Hence, this PDE yields an equivalent characterization of martingale observables when $g$ and $h$ are smooth. For non-smooth $g$ and $h$, such smoothness of martingale observables is generally \textit{not} true;\footnote{
		Under uniform ellipticity of the diffusion, classical parabolic regularity results often guarantee strong solutions. However, such diffusions are not our main focus here.
	} such martingale observables still exist (even continuous martingales), see Lemma~\ref{lem:Feynman--Kac mgale obsesrvable} (and Remark~\ref{remark:continuity}).
	
	As a na\"{i}ve but illustrative example of what can happen without condition~\eqref{eq:Hormander criterion}, let $X_t = (B^1_t, 0)$ be the one-dimensional Brownian motion embedded in $\mathbb{R}^2$. Now, any $f$ of the form $f(x, t) = f(x)=a(x_2)x_1+b(x_2)$ is a  martingale observable (with $g=h=0$), but not necessarily smooth. 
	
	\paragraph{\textbf{Generalized observables}}
	
	Let us point out a tiny generalization of Theorem~\ref{thm:mgale observables solve PDEs} that turned out to be useful in~\cite{ongoing}. In this subsection only, we let $f: \Lambda \times [0,T) \to \mathbb{C}$, $ g: \Lambda \times \Lambda \to \mathbb{C}$ and $ h: \Lambda \times [0,T) \to \mathbb{C}$, and we define (complex-valued) \textit{generalized $(g,h)$-martingale observables} as in Definition~\ref{def:mgale obs}, except with
	\begin{align*}
	\gamma_t = \gamma_{t_0, t } &:= \exp \left( \int_{s= t_0}^t g(X_s, X_{t_0}) \diff s \right).
	\end{align*}
	With cosmetic changes\footnote{A complex-valued martingale has real and imaginary parts that are real martingales, and the proof is based on a small-times limit, equating $X_s$ and $X_{t_0}$.} in our proof, one obtains:
	
		\begin{theorem}
		\label{thm:mgale observables solve PDEs 2}
		Fix $T>0$ and let $f: \Lambda \times [0,T) \to \mathbb{C}$ be locally bounded and Borel measurable, and $ g: \Lambda \times \Lambda  \to \mathbb{C}$ and $ h: \Lambda \times [0,T) \to \mathbb{C}$ continuous.
		If $f$ is a generalized $(g,h)$-martingale observable for the process $X$,
		and condition~\eqref{eq:Hormander criterion} is satisfied for that process, then $f$ is a weak solution to the complex PDE (two coupled real PDEs)
		\begin{align}
			\label{eq:thm-weak-pde}
			\mathcal{G} f(x, t)  + \partial_t f(x, t) + g(x,x) f (x, t) + h(x, t) = 0, \quad (x, t) \in \Lambda \times (0, T),
		\end{align}
		i.e., for all smooth, compactly-supported $\varphi: \Lambda \times (0, T) \to \mathbb{R}$, we have
		\begin{align}
			\label{eq:weak main PDE}
			\int_{x \in \Lambda} \int_{t=0}^T  \Big( \Re \big( f(x, t) \big) ( \mathcal{G}^*  - \partial_t ) \varphi (x, t) + \Re \big( f(x, t) g(x,x) +  h(x, t)  \big) \varphi (x, t) \Big) \diff t \diff^n x = 0,
		\end{align}
		and similarly for imaginary parts.
	\end{theorem}
	
	Note that Hörmander's theory is intrinsically real, and the smoothness of martingale observables is hence not guaranteed anymore. Also, even if they were smooth, due to the dependence of $g$ on $X_{t_0}$, the converse (any PDE solution being a generalized martingale observable) is \textit{not} generally true. For these reasons among others, Theorem~\ref{thm:mgale observables solve PDEs} is generally speaking a more natural setup. Nevertheless, such generalized martingale observables turn out to intrinsically appear in the mathematical-physics application of~\cite{ongoing}.
	
	\subsection{Feynman--Kac formulas}
	\label{subsec:FK}
	Classical Feynman--Kac formulas provide a stochastic representation of the (sometimes weak) solution to boundary value problems of the PDE~\eqref{eq:thm-weak-pde}, typically of the form  
	\begin{align}
		\label{eq:parabolic BVP}
		\begin{cases}
			\mathcal{G} f(x, t)  + \partial_t f(x, t) + g(x) f (x, t) + h(x, t) = 0, \qquad (x, t) \in \Lambda \times (0, T) \\
			\lim_{(w, s ) \to (x, t)} f(w, s) = \psi (x, t),  \qquad (x, t) \in   \partial \Lambda \times (0,T]  \; \cup \; \Lambda \times \{ T \},
		\end{cases}
	\end{align}
	where the boundary value function $\psi$ is given on the \textit{cylinder boundary} $\mathcal{C} := \partial \Lambda \times (0,T]  \; \cup \; \Lambda \times \{ T \} $.\footnote{
		If we start from the PDE~\eqref{eq:parabolic BVP}, with $a_{i,j}(x)$ smooth and the matrix $a$ positive semi-definite for all $x \in \Lambda$, then taking $\sigma(x) = a(x)^{1/2}$ to be the (unique) symmetric, positive semi-definite matrix square-root of $a(x)$ yields a process $X$ corresponding to this PDE. Indeed, then $a = \sigma \sigma^T$ by definition, and the smoothness of $\sigma$ follows by applying to a suitably scaled matrix $a$ the series expansion
		\begin{align}
			(I-M)^{1/2} = I-\tfrac{1}{2}M - \sum_{k=2}^\infty \tfrac{(2k-3)!!}{2^k k!}M^k,
		\end{align}
		valid for symmetric, positive semi-definite matrices $M$ with all eigenvalues below two.
	}
	As an application of Theorem~\ref{thm:mgale observables solve PDEs}, we formulate a multidimensional analogue which, to the best of our knowledge, widens the scope of such formulas, while also providing solutions in a strong sense.

	The solution candidate in Feynman--Kac theorems, and also here, is
	\begin{align}
		\label{eq: Feynman--Kac solution formula}
		f(x, t) := \mathbb{E} [\gamma_{t,\tau} \psi(X_\tau, \tau) + H_{t,\tau} \; |\; X_t = x],
		\qquad \text{where } \tau = \tau_{\partial \Lambda} \wedge T.
	\end{align}
	Due to the time-homogeneity of $X$,~\eqref{eq: Feynman--Kac solution formula} can be seen as ``launching the process from $x$ at the time $t$''. The function $f$ above is \textit{well-defined} if, for all launching points $(x, t) \in \Lambda \times [0, T)$, the processes $\gamma_{t, \cdot}$, $X_\cdot$ and $H_{t, \cdot}$ are all continuous up to \textit{and including} $\tau$ and the random variable inside the expectation is integrable. In that case,~\eqref{eq: Feynman--Kac solution formula} is almost manifestly a martingale observable (Lemma~\ref{lem:Feynman--Kac mgale obsesrvable}). By Theorem~\ref{thm:mgale observables solve PDEs}, $f(x,t)$ will thus in many cases provide smooth strong solutions to the PDE in~\eqref{eq:parabolic BVP}.
	
	Additionally, some boundary regularity of $\Lambda$ is needed to guarantee the desired boundary values for $f$; this already happens in the Brownian motion case and with continuous boundary values $\psi$. We will use the concept of $X$-regularity stemming from the analogous property of the Brownian motion. 
	\begin{definition}
		Assume that $X_t$ is continuous up to and including $\tau_{\partial \Lambda}$. We say that a point $x \in \partial \Lambda$ is \textit{$X$-regular}, if for any fixed $\delta > 0$, 
		\begin{align*}
			\mathbb{P}_w [|X_{\tau_{\partial \Lambda}} - x| < \delta \text{ and } \tau_{\partial \Lambda} < \delta ] \to 1 \qquad \text{as } w \to x,
		\end{align*}
		where $\mathbb{P}_w$ denotes the conditional probability given $X$ is launched from $w$.
		We say that $\Lambda$ has \textit{$X$-regular boundary} if every boundary point is $X$-regular.
	\end{definition}
	For instance, in $n=2$ dimensions and for $X$ being the two-dimensional Brownian motion, if $\partial \Lambda$ contains an isolated point, it is not $X$-regular (the harmonic measure of an isolated boundary point is zero in the plane); another classic example is the Lebesgue thorn for the three-dimensional Brownian motion.
	One could even construct a process where no boundary point is $X$-regular; an example are the slowed-down processes in Section \ref{subsec:slowed-down}.
	
	Our \emph{generalized} Feynman--Kac formula is the following.
	\begin{theorem}
		\label{thm:Feynman--Kac}
		Let $h: \Lambda \times [0, T ) \to \mathbb{R}$, $\psi: \mathcal{C} \to \mathbb{R}$, and $g: \Lambda \to \mathbb{R}$ be bounded continuous functions. If the process $X$ is continuous up to and including $\tau_{\partial \Lambda}$, the boundary of $\Lambda$ is $X$-regular, and condition~\eqref{eq:Hormander criterion} is satisfied for $X$,
		then the generalized Feynman--Kac formula~\eqref{eq: Feynman--Kac solution formula} gives a weak solution to~\eqref{eq:parabolic BVP}. If furthermore $g$ and $h$ are smooth, then $f$ is smooth and the unique bounded solution to~\eqref{eq:parabolic BVP}.
	\end{theorem}

	This is simultaneously an existence, uniqueness and regularity result, as well as a semi-explicit, stochastic solution formula. We deliver some remarks.
	
	\textit{1. Assumptions in the theorem} are comparatively general: one would hardly expect to prove a smooth Feynman--Kac type formula without, e.g., regular boundary  or H\"{o}rmander's condition.\footnote{
		In his original paper, H\"{o}rmander's argues that his condition is ``almost necessary'' to guarantee the smoothness of a generalized solution~\cite[Section~1]{Hormander-original_hypoellipticity_paper}.
	} We also note that even mild discontinuities in the boundary values $ \psi$ 
	(or the input functions $g, h$, at the cost of smoothness), such as $\psi(x, t) = \mathbb{I}\{ t = T\}$,
	can in many cases be handled by approximating $\psi$ from above and below by continuous functions and taking limits of equations~\eqref{eq: Feynman--Kac solution formula} and~\eqref{eq:weak main PDE}.
	
	\textit{2. Existence:}
	classical (non-stochastic) PDE theory guarantees the existence of strong solutions to \eqref{eq:parabolic BVP} provided that $\mathcal{G}$ is uniformly elliptic and $\Lambda$ has sufficiently regular boundary. 
	Beyond ellipticity, there are still powerful analysis techniques to solve PDEs such as the Hille--Yosida theorem. In this approach however, boundary values may still pose problems as one needs to verify assumptions on the operator with expanded domain that includes the boundary values. Hence our result could also have interest as an existence result (at the very least for probabilists, due to its convenient interpretation).
	
	\textit{3. Uniqueness:} it seems clear that at least some growth conditions are needed to guarantee uniqueness; we have not optimized our proof in this regard. Indeed, already for the one-dimensional (i.e., $\Lambda=\mathbb{R}$) heat equation (a time-reversal of~\eqref{eq:parabolic BVP}) Tychonoff gave the classic example~\cite{Tychonoff} of the non-uniqueness and a growth bound to guarantee the uniqueness. Note also that this discussion essentially only concerns unbounded domains $\Lambda$: for bounded $\Lambda$ and smooth $g, h$, by the hypoellipticity of $\mathcal{G} + \partial_t $ and the assumed continuous boundary values, any solution is continuous (or continuously extendable) in the set $\overline{\Lambda} \times [\epsilon, T]$ for any $\epsilon > 0$, and thus by compactness bounded in it. Uniqueness of the solution then follows.
	
	\textit{4. Regularity:} Our result can be contrasted to earlier ones that often only seem to give solutions in a ``stochastic sense'', see e.g.~\cite{Oksendal-SDE}.
	
	\subsection{$X$-harmonic problems}
	\label{subsec:intro X-harmonic}
	
	In this subsection, we illustrate our results by considering time-invariant PDEs of the form
	\begin{align}
		\label{eq:X-harmonic PDE problem}
		\begin{cases}
			\mathcal{G} f(x)  + g(x) f (x) + h(x) = 0, \qquad x \in \Lambda \\
			\lim_{w \to x} f(w) = \psi (x),  \qquad x \in   \partial \Lambda.
		\end{cases}
	\end{align}
	Following~\cite[Section~9]{Oksendal-SDE}, we call these $X$-harmonic problems, as they can be viewed as natural generalizations of the harmonic Dirichlet problem related to the Brownian motion case. We are still mainly concerned with \textit{non-elliptic} problems.
	We note that similar results (under the assumptions (a) or (b) below) have been studied in the literature: 
	an analysis account can be found in~\cite{Bony} and a stochastic account in~\cite{Stroock-Var}. 
	
	Let $\tau := \tau_{\partial \Lambda}$ throughout this subsection and set
	\begin{align}
		\label{eq:X-harmonic FK formula}
		f(x) := \mathbb{E}_{x} [\gamma_\tau \psi(X_\tau) + H_\tau]
	\end{align}
	which is \textit{well-defined} if $\tau < \infty$ a.s., $\gamma_t$, $X_t$ and $H_t$ are all continuous up to \textit{and including} $\tau$ and the random variable inside the expectation is integrable for all launching points $x \in \Lambda$. In this case we obtain the following:

	\begin{theorem}
		\label{thm:X-harmonic FK}
		Consider the PDE problem~\eqref{eq:X-harmonic PDE problem} with the functions $h, \psi$, and $g$ being bounded and continuous. Suppose that 
		$\tau : = \tau_{\partial \Lambda}$ is finite a.s.,
		the process $X$ is continuous up to and including $\tau$, the boundary of $\Lambda$ is $X$-regular, 
		and condition~\eqref{eq:Hormander criterion} is satisfied for $X$. Set  $C = \sup_{y \in \Lambda} g(y)$ and suppose also that at least one of the following holds:
		\begin{itemize}[noitemsep]
			\item[a)] $ C  < 0$; or
			\item[b)] $C=h=0$; or
			\item[c)] for some $\alpha > 0$,  $\mathbb{E}_w [\tau^\alpha e^{C \tau}] $ are finite for every $w$, and $\mathbb{E}_w [e^{C \tau}] \to 1$ as $w \to x \in \partial \Lambda$.
		\end{itemize}
		Then the function $f$ in~\eqref{eq:X-harmonic FK formula} is well-defined and a continuous weak solution to~\eqref{eq:X-harmonic PDE problem}. 
		If furthermore $g$ and $h$ are smooth, then $f$ is smooth and the unique bounded solution to~\eqref{eq:X-harmonic PDE problem}.
	\end{theorem}

	Compared to the moderate assumptions of previous theorem, assumptions (a)--(c) may seem restrictive. To illustrate their cruciality, consider a constant function $g(x)=S>0$, where $-S$ belongs to the spectrum of $\mathcal{G}$, i.e., there exists a smooth $\tilde{f} \neq 0$ such that $\mathcal{G} \tilde{f} + S \tilde{f} = 0$ and $\tilde{f}_{|\partial \Lambda} = 0$. Now, the uniqueness part of the theorem clearly ceases to hold true (for any $\psi$ and $h$ with at least one solution). To be completely explicit, this is the case for $X$ being the one-dimensional Brownian motion (so $\mathcal{G} = \tfrac{1}{2} \partial_{xx}$), $\Lambda = [-1, 1]$, $S=\pi^2/8$, and $\tilde{f}(x)=\cos(\tfrac{\pi x}{2})$. In particular, the only assumption violated in this example is (a)--(c). Furthermore, e.g., with $\psi = 0$ and $h=1$, and any $g(x)=S \geq \pi^2/8$, a more detailed computation with this example shows that $f$ is even not well-defined,
	see Section~\ref{subsec:large_C_example} for details.

	\section{Application examples}
	
	\subsection{Applications to SLEs}
	\label{subsec:SLE}
	
	In SLE language, we prove in Proposition~\ref{prop:SLE martingales observables are smooth} below that \textit{all conformally covariant local martingale observables of SLE type curves are smooth}. Results of this type appear at least in~\cite{Dubedat, PW, AHPY, FLPW}; we however hope that our exposition has value as a ``black box'' with a general but simple statement, especially for further results such as~\cite{FW, HPW}. Furthermore, our result will analogously be useful in an ongoing project with more general (non-covariant) SLE observables~\cite{ongoing}.\footnote{The proof in the simpler, covariant case in~\cite{AHPY} relies on a change of variables that cleverly removes order zero terms from the PDE corresponding to covariant variables; such a trick, and an analogous proof, is generally (likely) not at hand.}

	
	\subsubsection*{\textbf{Definition of SLE}}
	
	Schramm--Loewner evolution (SLE) type curves are conformally invariant random curves \cite{Schramm-LERW_and_UST, RS-basic_properties_of_SLE}
	that are known or conjectured to describe (scaling limits of) random interfaces in many critical planar models~(e.g.~\cite{Smirnov-critical_percolation, LSW-LERW_and_UST, 
		Zhan-scaling_limits_of_planar_LERW, HK-Ising_interfaces_and_free_boundary_conditions,
		CDHKS-convergence_of_Ising_interfaces_to_SLE}).
	
	There are several SLE type models; for definiteness, we only consider here those defined in the upper half-plane $\mathbb{H}$ via the Loewner differential equation
	\begin{align}
		\label{eq:Loewner ODE}
		\partial_t g_t (z) &= \tfrac{2}{g_t(z)-X^1_t}, \qquad g_0 (z) = z,
	\end{align}
	where $X^1_\cdot: \mathbb{R}_{\geq 0} \to \mathbb{R}$ is some continuous function, and the solution to this ordinary differential equation with a given starting point $ z $ in\footnote{We will allow $z \in \overline{\mathbb{H}}$ without explicit mention whenever it is more beneficial.} $ \mathbb{H}$ is only defined up to the (possibly infinite) explosion time when $g_t(z)$ and $X^1_t$ collide. The sets $K_t$ where the solution is not defined up to time $t$ are those carved out by the initial segment of the SLE curve, while $g_t $ is\footnote{
		We use the notation $g_t(z)$ for the conformal maps in this subsection, in order to comply with conventional SLE notation. It should not be confused with the ``killing/reproduction rate'' $g(x)$ in the rest of this article.
	} a conformal map $\mathbb{H} \setminus K_t \to \mathbb{H}$. See~\cite{Lawler-SLE_book} for an introduction.
	
	The most common SLE variant, the chordal $\mathrm{SLE}(\kappa)$ with parameter $\kappa \geq 0$ from $0$ to $\infty$ in $\mathbb{H}$ is obtained by taking $X^1_t = \sqrt{\kappa} B_t$.
	With the same effort, we treat below SLE variants that also depend on $X^i_t := g_t(x_i)$ for some $x_2, \ldots, x_n \in \mathbb{R}$ and that are, up to the exit time of some domain $\Lambda \subset \mathbb{R}^n$ by $X_t$, of the form\footnote{We assume that $\Lambda$ does not allow coordinate collisions $x_j = x_1$ for any $j \neq 1$.}
	\begin{align}
		\notag
		\diff X^1_t &= \sqrt{\kappa} \diff B_t + b_1 (X_t ) \diff t \\
		\label{eq:SLE SDE}
		\diff X^i_t &= \tfrac{2 \diff t}{ X^i_t - X^1_t}, \qquad 2 \leq i \leq n,
	\end{align}
	where a smooth function $b_1: \Lambda \to \mathbb{R}$ and the launching points $(x_1, \ldots, x_n) = X_0$ are given. In this note, we will only treat SLEs through such SDEs.
	
	\subsubsection*{\textbf{Conformally covariant SLE martingale observables are smooth}}
	
	Note that the diffusion~\eqref{eq:SLE SDE} above is \textit{not} elliptic. Also, SLE type processes will in many applications \textit{not} have infinite lifetime, and the SDEs have unbounded terms. This complicates analysis via both PDE and SDE theory and highlights the use of our results. The H\"{o}rmander condition for SLE is well-known~(at least~\cite{Dubedat, AHPY, FLPW, mie2}).
	
	\begin{lemma}
		\label{lem:SLE Hormander}
		For the process $X_t$ defined by~\eqref{eq:SLE SDE}, the Lie algebra generated by the corresponding operators $\{ \mathcal{U}_i \}_{i=1}^q$ and $\{ [\mathcal{U}_i, \mathcal{U}_0] \}_{i=1}^q$ defined by~\eqref{eq:Lie generators 1}--\eqref{eq:Lie generators 2} is of dimension $n$ at every point $x \in \mathbb{R}^n$ with $x_i \neq x_j$ for all $i \neq j$.
	\end{lemma}
	
	\textit{Conformally covariant (time-independent) SLE observables} are of the form
	\begin{align}
		\label{eq:def of conf cov SLE observable}
		M_t := \Big( \prod_{i=2}^n (g_t'(x_i)) ^{\Delta_i} \Big) f(X_t)
	\end{align}
	for some locally bounded Borel measurable $f: \Lambda \to \mathbb{R}$, and constants $\Delta_i \in \mathbb{R}$ called the \textit{conformal weights} at $x_i$, $2 \leq i \leq n$.
	Note that~\eqref{eq:SLE SDE} gives
	\begin{align*}
		(g_t'(x_i)) ^{\Delta_i} = \exp \big(  \int_{s= 0}^t -  \tfrac{ 2 \Delta_i}{(X^i_s - X^1_s)^2} \diff s \big).
	\end{align*}
	The next result is a direct consequence of Theorem~\ref{thm:mgale observables solve PDEs} and Lemma~\ref{lem:SLE Hormander}.
	\begin{proposition}
		\label{prop:SLE martingales observables are smooth}
		If a conformally covariant SLE observable $M_t$ as in~\eqref{eq:def of conf cov SLE observable} is
		a local martingale under the process~\eqref{eq:SLE SDE}, then $f$ is a smooth solution to
		\begin{align}
			\label{eq:SLE gen PDE}
			\bigg( \frac{\kappa}{2} \frac{\partial^2}{ \partial x_1^2}  + \sum_{i=2}^n \frac{2}{x_i - x_1} \frac{\partial}{ \partial x_i} + b_1(x) \frac{\partial}{ \partial x_1} -  \sum_{i=2}^n \frac{2 \Delta_i }{(x_i - x_1)^2} \bigg) f(x)  = 0.
		\end{align}
	\end{proposition}

	\begin{remark}
	SLE martingale observables depending on non-real marked points are equally important. Since the Hörmander theory is intrinsically real, their real and imaginary parts should then be treated as separate real coordinates. Technical verification of Condition~\eqref{eq:Hormander criterion} is however turned into a complex algebra question by the following observations: suppose that in our setup $X^1_t, X^2_t$ are finite-variation processes, i.e.,
	\begin{align*}
	\diff X^i_t = b_i (X_t) \diff t, \qquad i \in \{ 1, 2 \}.
	\end{align*}
obtained from a complex process $Z_t = X^1_t + i X^2_t$. That is, with $\beta (x) = b_1 (x) + i b_2 (x)$ we have originally had
\begin{align*}
 \diff Z_t = \beta (X_t) \diff t.
\end{align*}
Now, note first that due to the finite variation, in the operators $\mathcal{U}_0$ and $\mathcal{U}_q$ in Condition~\eqref{eq:Hormander criterion}, the derivatives $\partial_1, \partial_2$ only appear in $\mathcal{U}_0$ in the terms
\begin{align*}
b_1 (x) \partial_1 + b_2 (x) \partial_2 = \beta (x) \partial_z + \overline{\beta (x)} \overline{\partial_z},
\end{align*}
where $ \partial_z = \tfrac{1}{2}(\partial_1 - i \partial_2)$ and $ \overline{\partial_z} = \tfrac{1}{2}(\partial_1 + i \partial_2)$ are the usual Wirtinger derivatives. In other words, the \textit{real} operators $\mathcal{U}_0$ and $\mathcal{U}_q$ (and their Lie brackets) can be expressed in terms of complex linear combinations of the complex operators in $z$. Secondly, simple linear algebra arguments show that
\begin{align*}
& \text{The Lie algebra generated by } \{ \mathcal{U}_q \}_{q=1}^d \\
		\label{eq:Hormander criterion}
		\tag{H}
		&  \text{and } \{ [\mathcal{U}_q, \mathcal{U}_0] \}_{q=1}^d \text{ is of dimension $n$ at a given point $x \in \Lambda$.} \\
 \Leftrightarrow 
 & \partial_1, \ldots, \partial_n \text{ can all be expressed as finite real linear combinations} \\
 & \text{of $\{ \mathcal{U}_q \}_{q=1}^d$ and composed Lie bracket operations involving $\{ \mathcal{U}_r \}_{r=0}^d$ at $x$} \\
  \Leftrightarrow 
 & \partial_z , \overline{\partial_z}, \partial_1, \ldots, \partial_n \text{can all be expressed as finite \textit{complex} linear combinations} \\
 & \text{ of $\{ \mathcal{U}_q \}_{q=1}^d$ and composed Lie bracket operations involving $\{ \mathcal{U}_r \}_{r=0}^d$ at $x$.}
\end{align*}
So in a nutshell, the vector fields and the Hörmander criterion can be expressed and verified in terms of the Wirtinger derivatives $\partial_z , \overline{\partial_z}$ instead of $\partial_1 , \partial_2$.
	\end{remark}
	
	We now point out interpretations and concrete applications of this result.
	
	\subsubsection*{\textbf{PDEs of Conformal field theory}}
	
	In Conformal field theory, i.e., the representation-theoretic (conjectural) description of limits of critical models by physicists, the PDE~\eqref{eq:SLE gen PDE} (with $b_1 = 0$) is known as the Belavin--Polyakov--Zamolodchikov (BPZ) PDE (with conformal weights $\Delta_i$ at $x_i$). Concluding the converse implication to the theorem by It\^{o} calculus, we notice that \textit{the BPZ PDEs equivalently characterize SLE local martingale observables}.
	
	\subsubsection*{\textbf{Variants of SLE convergence results}}
	
	Recall that one of the main motivations behind SLE theory is its relation to lattice models.
	It often happens that one studies two families of measures $\mathbb{P}_\epsilon$ and $\mathbb{Q}_\epsilon$, indexed by lattice mesh sizes $\epsilon > 0$, where some lattice curves under $\mathbb{P}_\epsilon$ are known to converge to an SLE type process, while $\mathbb{Q}_\epsilon$ is some mild variation of $\mathbb{P}_\epsilon$, e.g. a conditioning on an event $A$. Then, by a discrete Girsanov argument, $\frac{ \diff \mathbb{Q}_\epsilon }{\diff \mathbb{P}_\epsilon} |_{\mathcal{F}_t}$ is a $\mathbb{P}_\epsilon$-martingale (e.g., if $\mathbb{Q}_\epsilon$ is the $A$-conditional measure, then $\frac{ \diff \mathbb{Q}_\epsilon }{\diff \mathbb{P}_\epsilon} |_{\mathcal{F}_t} = \frac{\mathbb{P}_\epsilon[ A \; |\; \mathcal{F}_t ]}{\mathbb{P}_\epsilon [A]}$). Now, with reasonable control of this martingale as $\epsilon \to 0$, one is sometimes able to conclude that it converges to an SLE martingale observable. By the above theorem, the limit measure of $\mathbb{Q}_\epsilon$ on $\mathcal{F}_t$ is thus obtained by weighting the limit measure of $\mathbb{P}_\epsilon$ with a smooth SLE martingale (a solution to the BPZ PDE), and with the (usual) Girsanov's theorem (as explicated below), one can actually characterize the limit of $\mathbb{Q}_\epsilon$:s as an SLE type curve. A smoothness of martingale observables result is used in such a context at least in~\cite{FW}.

	\subsubsection*{\textbf{Construction of multiple SLEs}}
	
	Consider the usual chordal SLE$(\kappa)$ measure $\mathbb{P}$ with $b_1 = 0$ in~\eqref{eq:SLE SDE}, and let us follow the evolution of $n=2N$ boundary points by the SDEs~\eqref{eq:SLE SDE}.
	The \textit{local multiple SLE} of $N$ curves between $X_0^1, \ldots, X^{2N}_0$ is another measure $\mathbb{Q}$ obtained by weighting $\mathbb{P}$ with a non-negative local\footnote{Such weighting relies on martingales and thus for the rest of this subsection, without explicated notation, all processes are only considered up to a suitable stopping time.}
	martingale observable
	$\frac{ \diff \mathbb{Q} }{\diff \mathbb{P}} |_{\mathcal{F}_t} = M_t/M_0$, where $M_t  $ is of the form~\eqref{eq:def of conf cov SLE observable}, with $n=2N$ variables,
	$\Delta_i = \frac{6-\kappa}{2 \kappa}$ for all\footnote{
		The present result may in future be of interest for other SLE variants with other values of $\Delta_i$; e.g., such have been constructed directly for $\kappa=2$ in~\cite[Appendix~D]{KLPR}. For this reason, we do not substitute the value of $\Delta_i$ into the formulas.
	}
	$2 \leq i \leq 2N$, and with $f$ satisfying the conformal covariance condition that
	\begin{align*}
		f(x ) = \big( \prod_{i=1}^{2N} \phi'(x_i)^{\Delta_i} \big) f (\phi (x_1), \ldots, \phi (x_{2N})) 
	\end{align*}
	for all conformal (Möbius) maps $\phi: \mathbb{H} \to \mathbb{H}$ that preserve the order of the marked points on the real line.
	
	In prior literature, it has been a convention to require that additionally $f$ is at least twice continuously differentiable (and hence the local-martingaleness of $M_t$ has been stated in form of the PDE~\eqref{eq:SLE gen PDE}). By Proposition~\ref{prop:SLE martingales observables are smooth}, such smoothness is however automatic, so the above definition in terms of local martingales indeed is equivalent.

	Constructions of multiple SLEs that exist in the literature go through martingale observables,\footnote{
		There exist other ways for certain special values of $\kappa$ (e.g.,~\cite{Izyurov-critical_Ising_interfaces_in_multiply_connected_domains, KKP, PW}) and for $N=2$~\cite{Wu-hypergeometric}.
	}
	and there have been several recent expositions of proving or referring to the smoothness~\cite{Dubedat, PW, Wu-hypergeometric,FLPW, Zhan,AHPY}. We hope that our general and simple statement clarifies the discussion on this result, and, in the future, on analogues for other SLE variants (cf.~\cite{HPW}).
	
	
	
	Let us illustrate in one paragraph the cruciality of such $f$ being smooth. With Girsanov's theorem, we can find the law of $X^1$ under $\mathbb{Q}$:
	\begin{align*}
		\diff X^1_t = \sqrt{\kappa} \diff \tilde{B}_t + \diff \langle X^1, L \rangle_t,
	\end{align*}
	where $\tilde{B}$ is a $\mathbb{Q}$-Brownian motion and $L$ is the stochastic logarithm of the measure-transforming local martingale $M_t/ M_0$, i.e., $L_t = \int_{s=0}^t \tfrac{ \diff M_t}{M_t}$. Now, an overwhelming majority multiple SLE results rely on $f$ furthermore being smooth, in which case It\^{o}'s theorem yields
	\begin{align*}
		\diff X^1_t = \sqrt{\kappa} \diff \tilde{B}_t + \kappa \tfrac{\partial_1 f (X_t)}{f(X_t)} \diff t,
	\end{align*}
	and still $\diff X^i_t = \tfrac{2 \diff t}{X^i_t - X^1_t}$. Some examples of such results are:
	\begin{enumerate}
		\item Classification of the convex set of multiple SLE measures: a collection such measures are convex-independent if and only if the corresponding PDE solutions $f$ are linearly independent~\cite{mie2}. This explicates a probabilistic connection to analysis of the PDE solution space, e.g.~\cite{FK-solution_space_for_a_system_of_null_state_PDEs_1, FK-solution_space_for_a_system_of_null_state_PDEs_2, FK-solution_space_for_a_system_of_null_state_PDEs_3, FK-solution_space_for_a_system_of_null_state_PDEs_4, PW, FLPW}.
		\item The limits of certain topological pairing probabilities in the several lattice models converging to SLEs were solved through an SLE argument and a detailed analysis of $f$~\cite{PW, PW18, FPW, mie2}.
		\item Construction of other SLE type models often rely on such expressions (e.g.~\cite{JJK-SLE_boundary_visits, mie3, KLPR}).
	\end{enumerate}


	
	\subsection{Some verifications of the assumptions of Theorem \ref{thm:X-harmonic FK}}
	\label{subsec:large_C_example}
	
	In the case $C \geq 0$, $h \neq 0$, it might be difficult to verify condition (c) of Theorem~\ref{thm:X-harmonic FK}. We remark the following:
	
	\begin{lemma}
		Suppose that $X$ is continuous up to and including its boundary hitting time $\tau_{\partial \Lambda} = \tau$, and that for all $s > 0$ (resp. for $s=0$) the PDE boundary value problem
		\begin{align}
			\label{eq:X-harmonic life-time problem}
			\begin{cases}
				\mathcal{G} f(x) -sf(x)  + 1 = 0, \qquad x \in \Lambda \\
				\lim_{w \to x} f(w) = 0,  \qquad x \in   \partial \Lambda,
			\end{cases}
		\end{align}
		has a bounded smooth solution $f$.
		Then $f = f_s$ is related to $\tau$ via the Laplace transform $ \mathbb{E}_x[\mathbb{I}\{\tau < \infty\} e^{-s\tau}]  = -sf_s(x) + 1$. (Resp. then $\tau$ is a.s. finite and integrable and $f(x) = \mathbb{E}_x[\tau]$, and assumption (c) of Theorem \ref{thm:X-harmonic FK} holds for $C=0$ and $\alpha = 1$.)
	\end{lemma}
	
	For a proof, in the case $s=0$, observe (with straightforward It\^{o} calculus) that $f(X_t)+t$ is a martingale, and hence $f(x) = \mathbb{E}_x [f(X_{\tau \wedge T}) + \tau \wedge T] \geq -\sup |f| + \mathbb{E}_x [\tau \wedge T]$. Due to this bound on $\mathbb{E}_x [\tau \wedge T]$, by Monotone convergence as $T \to \infty$, $\tau$ is finite and integrable. Dominated convergence then gives $f(x)=\mathbb{E}_x [ \tau ]$. The case $s>0$ is similar (but even simpler); the martingale to look at is $ e^{-st}f(X_t) + \tfrac{1-e^{-st}}{s}$.

	To sketch an explicit example, let $X$ be a one-dimensional Brownian motion and $\Lambda = [-1, 1]$.  Then one gets $f_s(x)=\tfrac{1}{s}(1-\tfrac{\cosh(x\sqrt{2s})}{\cosh(\sqrt{2s})} )$, and  by Monotone convergence, $\mathbb{E}_x[ \tau^k ] = (-1)^k \lim_{s \downarrow 0} \partial_s^k \mathbb{E}_x[ e^{-s\tau}]$ as an equality in $\mathbb{R} \cup \{ + \infty \}$. Thus, we have
	\begin{align*}
		\mathbb{E}_{x=0} [\tau^k] &= (-1)^k k! 2^k a_{2k}, \quad \text{where} \quad
		1/\cosh(x) = \sum_{k=0}^\infty a_{2k} x^{2k},
	\end{align*}
	and we obtain $\mathbb{E}_0 [e^{C \tau}] = 1/\cosh(i\sqrt{2C}) = 1/\cos(\sqrt{2C})$ for $i\sqrt{2C}$ within convergence disc of the above power series, i.e., $C < \pi^2/8$, and similarly $\mathbb{E}_0 [e^{C \tau}] = + \infty$ for $C \geq \pi^2/8$. In particular, criterion (c) of Theorem~\ref{thm:X-harmonic FK} is readily concluded for $C < \pi^2/8$, while, as observed in the Introduction, the theorem is not true for $C \geq \pi^2/8$.

	\subsection{PDE connections for processes with finite life-time}
	
	In this subsection, we consider processes $X$ killed upon the hitting time $\tau_{\partial \Lambda} = \sup \{ t \geq 0 : \; X_t \in \Lambda \}$. Note that the coefficients of the SDE~\eqref{eq:main process of interest} are allowed to blow up at $\partial \Lambda$, and $X_{\tau_{\partial \Lambda}}$ need not even be well-defined.\footnote{
		Formally, we can for instance add a ``cemetery point'' $\dagger$ and set $X_t = \dagger$ for $t \geq \tau_{\partial \Lambda}$.
	} The following weaker boundary regularity will turn out to be crucial:
	
	\begin{definition}
		We say that $x \in \partial \Lambda$ is \textit{weakly $X$-regular}, if for any $\delta > 0$,
		\begin{align*}
			\mathbb{P}_w [ \tau_{\partial \Lambda} < \delta ] \to 1 \qquad \text{as } w \to x.
		\end{align*}
		We say that $\Lambda$ has \textit{weakly $X$-regular boundary} if every boundary point is weakly $X$-regular.
	\end{definition}
	
	\subsubsection*{\textbf{The survival of $X$}}
	
	\begin{proposition}
		\label{prop:survival proba}
		Fix $T > 0$ and let $f: \Lambda \times [0,T] $ be the survival probability
		\begin{align*}
			f(x, t) := \mathbb{P}_{x, t} [\tau_{\partial \Lambda} > T].
		\end{align*}
		If $\Lambda$ has weakly $X$-regular boundary and
		condition~\eqref{eq:Hormander criterion} is satisfied for $X$, then $f$ is a smooth solution to the boundary value problem
		\begin{align}
			\label{eq:survival proba PDE}
			\begin{cases}
				\mathcal{G} f(x, t)  + \partial_t f(x, t) = 0, \qquad (x, t) \in \Lambda \times (0, T) \\
				\lim_{(w, s ) \to (x, t)} f(w, s) = \mathbb{I} \{t=T\}, \qquad (x, t) \in \mathcal{C}.
			\end{cases}
		\end{align}
	\end{proposition}
	
	The proof is almost trivial: $f(X_t, t)$ is manifestly a martingale, and by Theorem~\ref{thm:mgale observables solve PDEs}, $f$ is a smooth solution to the above PDE. The boundary values hold by weak boundary regularity.
	
	\begin{corollary}
		Let $f$ be as above, assume $f(w, t) > 0$ for all $(w, t) \in \Lambda \times [0,T]$, and let $\mathbb{Q} $ be the probability measure $\mathbb{Q} [\, \cdot \,] := \mathbb{P} [\, \cdot \, |\, \tau_{\partial \Lambda} > T ] $.
		Then,
		\begin{align}
			\label{eq:conditional SDE}
			\diff X_t = \sigma (X_t) \diff \tilde{B}_t + a (X_t)  \tfrac{\nabla f (X_t, t)}{ f (X_t, t) } \diff t + b (X_t)  \diff t,
		\end{align}
		where we identified vectors as column matrices and $\tilde{B}$ is a $d$-dimensional $\mathbb{Q}$-Brownian motion.
	\end{corollary}
	
	Also this proof is trivial; we use $\tfrac{\diff \mathbb{Q}}{ \diff \mathbb{P}}|_{\mathcal{F}_t} = f(X_t, t)/f(X_0, 0)$ and Girsanov's theorem and basic It\^{o} calculus to transform the Brownian motions. Note that the smoothness of $f$ from Theorem~\ref{thm:mgale observables solve PDEs} is crucial.
	
	An interesting question is whether the SDE~\eqref{eq:conditional SDE} turns time-independent for $t \in [0,1]$ as $T \to \infty$; this would be the process $X$ conditioned to stay inside $\Lambda$ forever. If so happens, then $f(x, t)$ must separate asymptotically, $f(x, t) \approx \alpha(x) \beta(t)$. Assuming that a solution of the form $f(x, t) = \alpha(x) \beta(t)$ to the PDE in~\eqref{eq:survival proba PDE} exists, one then readily shows that $\beta (t) = e^{st}$ for some $s \geq 0$ and $\mathcal{G} \alpha + s \alpha = 0$. The smallest $s \geq 0$ then dominates. Perhaps one could even have a series expansion for $f$ in terms of such $\alpha_i (x) e^{s_i t}$ with different eigenfunctions $\alpha_i$ corresponding to eigenvalues $s_i$ (as in a Fourier series solution of the heat eqaution). This also sheds light on the previous example.

	\subsubsection*{\textbf{Transition probabilities}}
	
	We now point out a proof of Kolmogorov's PDEs with minimal assumptions. Considering the extent of literature on the topic, it is not impossible that the result would follow from (a mild generalization of) some earlier exposition. 
	We are however not aware of such a reference, the crucial point being that we avoid the technical difficulties that would arise when defining a dual Markov semigroup for an irregular diffusion, as in the usual proof of the backward equation (e.g.,~\cite{Stroock-PDE}).
	
	\begin{proposition}
		\label{prop:killed}
		Suppose that 
		condition~\eqref{eq:Hormander criterion} is satisfied for $X$.
		Then for each fixed $t > 0$, $X_t$, with $X_0 = w \in \Lambda$, allows a density function\footnote{
			This density generates a sub-probability measure, whose total mass $\int \rho_w(x,t) \diff^n x$ equals the survival probability $f(w, T- t) \leq 1$ of the Proposition~\ref{prop:survival proba} (for auxiliary $T \geq t$).
		} $\rho_w(x,t)$ with respect to the Lebesgue measure on $\Lambda$,
		and $\rho_\cdot (\cdot, \cdot)$ is a smooth function on $\Lambda \times \Lambda \times (0, \infty)$, that solves 
		\begin{align}
			\label{eq:particle density BVP}
			\partial_t \rho_w (x, t) = \mathcal{G}^*\rho_w(x, t) 
			\qquad \text{and} \qquad
			\partial_t \rho_w(x, t) = \mathcal{G}_w \rho_w(x, t).
		\end{align}
	\end{proposition}
	\begin{remark}
		One is also naturally interested in the boundary values associated with the above PDEs. However, even in the one-dimensional case, this is a highly non-trivial matter (e.g., \cite[Chapter II.1.6]{Borodin-Salminen}).
	\end{remark}
	
	\begin{proof}[Proof of Proposition~\ref{prop:killed}]
		The PDEs and smoothness are proven simultaneously via H\"{o}rmander theory: in the language of Section~\ref{subsec:PDE background}, we will show that the distribution $\mathfrak{u}$ on compactly-supported smooth functions $\varphi: \Lambda \times \Lambda \times [0, T] \to \mathbb{R}$ given by
		\begin{align*}
			\mathfrak{u}(\varphi) := \int_{w \in \Lambda}\int_{t=0}^T \mathbb{E}_w [\varphi(w, X_t, t) \mathbb{I}\{ X_t \in \Lambda \}] \diff t \diff^n w
		\end{align*}
		is a distributional solution to the PDEs~\eqref{eq:particle density BVP}.
		We may assume below that $\varphi(w, x, t)$ is a probability density function.
		
		Then, for the forward PDE, write  $\varphi(w, x, t) = \chi(w) \psi_w (x, t)$ in terms of a marginal and conditional density $\chi$ and $\psi_w$, and interpret $\psi_w(\dagger, t) := 0$ below. By compact supports, It\^{o}'s formula and Fubini,
		\begin{align*}
			0 &= \mathbb{E}_w [\psi_w(X_T, T) - \psi_w (X_0, 0) ] 
			\\
			&= \int_{t=0}^T\mathbb{E}_w \Big[ \big( \mathcal{G} \psi_w (X_t, t) + \partial_t \psi_w (X_t, t) \big) \mathbb{I}\{  X_t \in \Lambda  \} \Big] \diff t,
		\end{align*}
		from which one readily obtains distributional forward PDE $\mathfrak{u}((\mathcal{G} + \partial_t)\varphi) = 0$.
		
		For the backward PDE, suppose that $\varphi(w, x, t)$ is of the form $\chi(w, t)\psi(x)$. Not all density functions are of this form, but with linear combinations of such functions with the different $\psi$:s being a partition of unity in terms of smooth bump functions, any $\varphi$ and its $w$ and $t$ derivatives up to order two can be approximated to an arbitrary precision (e.g., in the sup norm).
		Hence, it is still sufficient to consider such form of the density $\varphi$.
		Now, set $\psi(\dagger) := 0$ and $\tilde{f}(x, t) := \mathbb{E}_{x, t}[\psi (X_T)] = \mathbb{E}_{x}[\psi (X_{T-t})].$
		By definition, $\tilde{f}(X_t, t)$ is a martingale and by Theorem~\ref{thm:mgale observables solve PDEs} a smooth solution to the PDE $\mathcal{G} \tilde{f} + \partial_t \tilde{f} = 0$. Integrating $(\mathcal{G} \tilde{f} + \partial_t \tilde{f})$ against the test function $\chi (w, T-t)$, we obtain
		\begin{align*}
			0
			&= 
			\int_{w \in \Lambda} \int_{t=0}^T   \big((\mathcal{G}^*_w -\partial_t) {\chi} (w, T- t) \big) \mathbb{E}_{w}[\psi(X_{T-t})] \diff^n w \diff t \\
			&= \int_{w \in \Lambda} \int_{t=0}^T \mathbb{E}_{w} \big[ \psi(X_{t}) (\mathcal{G}^*_w + \partial_t) {\chi} (w, t) \big] \diff^n w \diff t,
		\end{align*}
		which is the distributional formulation $\mathfrak{u}((\mathcal{G}_w^* + \partial_t)\varphi) = 0$ of the backward PDE. By H\"{o}rmander's theorem (Theorem~\ref{thm:Hormander original}; see also Remark~\ref{rem:different Lie algebras}), $\mathfrak{u}$ is then represented by a smooth integration kernel $\rho_w(x, t)$, which solves the forward and backward PDEs; it is readily argued to be the particle density.
		
	\end{proof}
	
	\begin{remark}
		With an analogous proof, one can show that the distribution
		\begin{align}
			\label{eq:def of distribution v}
			\mathfrak{v}(\varphi) := \int_{w \in \Lambda}\int_{t=0}^T \mathbb{E}_w [\gamma_t \varphi(w, X_t, t) \mathbb{I}\{ X_t \in \Lambda \}] \diff t \diff^n w,
		\end{align}
		if the integrals are well-defined,
		solves under criterion~\eqref{eq:Hormander criterion}, the PDEs 
		\begin{align}
			\label{eq:distributional PDE with order zero terms}
			( \mathcal{G}^* + g(x) - \partial_t ) \mathfrak{v} = 0 \qquad \text{and} \qquad
			( \mathcal{G}_w + g(w) - \partial_t ) \mathfrak{v} = 0.
		\end{align}
	\end{remark}
	
	\section{Proofs}
	\label{sec:proofs}
	\subsection{Preliminaries}
	\label{subsec:PDE background}
	
	This subsection reviews well-known results.
	
	\subsubsection*{\textbf{Hörmander's criterion}}
	
	Let $\mathcal{C}^\infty_c(\Lambda)$ denote the set of compactly-supported smooth functions $\Lambda \to \mathbb{R}$. A \textit{distribution} is a continuous linear map $\mathcal{C}^\infty_c(\Lambda) \to \mathbb{R}$ (the topology on $\mathcal{C}^\infty_c(\Lambda )$ can be found in, e.g.,~\cite{Rudin-FA}). Given a smooth linear differential operator on $\Lambda$,
	\begin{align*}
		\mathcal{L} = \sum_{\alpha \in A} c_\alpha(x) D^\alpha,
	\end{align*}
	where the sum ranges over finite subset $A$ of multi-indices $\alpha$, $c_\alpha: \Lambda \to \mathbb{R}$ are smooth and $D^\alpha$ denotes the corresponding derivative, define the dual operator as
	\begin{align*}
		(\mathcal{L}^* f)(x) = \sum_{\alpha \in A} (-1)^\alpha D^\alpha \big( c_\alpha(x) f(x) \big).
	\end{align*}
	Now, a distribution $\mathfrak{u} \in \mathcal{C}^\infty_c(\Lambda)^*$ is said to be a \textit{distributional solution} to the PDE $\mathcal{L}\mathfrak{u} = h $, where $h: \Lambda \to \mathbb{R}$ is a given, sufficiently regular (e.g.,  $L^2_{\mathrm{loc}}$) function, if $ \mathfrak{u} (\mathcal{L}^* \varphi ) = \int_{x \in \Lambda} h(x) \varphi (x) \diff^n x $ for all $\varphi \in \mathcal{C}^\infty_c(\Lambda)$. Finally, the operator $\mathcal{L}$ is said to be \textit{hypoelliptic} if any distributional solution to $\mathcal{L} \mathfrak{u} = h$, for any smooth $h: \Lambda \to \mathbb{R}$, is of the form
	\begin{align*}
		\mathfrak{u} (\varphi) = \int_{x \in \Lambda} u(x) \varphi(x) \diff^n x \qquad \text{for all } \varphi \in \mathcal{C}^\infty_c(\Lambda)
	\end{align*}
	for some (in that case unique) smooth function $u: \Lambda \to \mathbb{R}$. It immediately follows that $u$ is a strong solution to the PDE  $\mathcal{L} u = h$. Hence, in a nutshell, hypoellipticity means ``any distributional solution is a smooth strong solution''. As a particular example, all uniformly elliptic operators arising from second order PDEs are also hypoelliptic. 
	
	We now present Hörmander's sufficient criterion for hypoellipticity, in the generality of PDEs on open subsets $\Lambda \subset \mathbb{R}^n$; see, e.g.,~\cite[Section~7.4]{Stroock-PDE}.
	
	Let $\mathcal{L}$ be a second-order differential operator on $\Lambda \subset \mathbb{R}^n$, of the form
	\begin{align}
		\label{eq:Hormander's 2nd order operator}
		\mathcal{L} = \tfrac{1}{2} \sum_{q=1}^d \mathcal{V}^2_q + \mathcal{V}_0 + g(x),
	\end{align}
	where $g: \Lambda \to \mathbb{R}$ is smooth, $\mathcal{V}_q$ for $0 \leq q \leq d$ are smooth vector fields (i.e., first-order smooth differential operators on functions $\Lambda \to \mathbb{R}$, with no order zero term), and $\mathcal{V}^2_q$ denotes composition of such operators, i.e., $(\mathcal{V}^2_q f )(x) = \mathcal{V}_q(\mathcal{V}_q f (x))$.
	
	The Lie bracket $[\cdot, \cdot]$ of smooth vector fields is defined in the usual manner and provides another smooth vector field. The Lie algebra generated by some vector fields at a given point $x \in \Lambda$ is the span of those operators and any other operators obtained by iterated Lie brackets from the original ones. The dimension of this Lie algebra at a given coordinate $x$ is obtained by interpreting its elements as vectors in $\mathbb{R}^n$ by picking the coefficients (evaluated at $x$) of the derivatives $\partial_1, \ldots, \partial_n$ in each operator. Hörmander's famous theorem can now be stated as follows.
	
	\begin{theorem}
		\label{thm:Hormander original}
		If the dimension of the Lie algebra generated by $\mathcal{V}_0, \ldots, \mathcal{V}_d$ is $n$ in every point $x \in \Lambda$, then the linear operator $\mathcal{L}$ in~\eqref{eq:Hormander's 2nd order operator} is hypoelliptic.
	\end{theorem}
	
	Note that the criterion is independent of the order zero term $g(x)$ in the PDE, as long as $g$ is smooth.
	
	\subsubsection*{\textbf{Particle densities and Kolmogorov's PDEs}}

	Coming back to our application,
	provided that the random variables $(X_t)_{t > 0}$ have sufficiently regular density functions $\rho(x,t)$ over $\mathbb{R}^n$ (with auxiliary law of $X_0 \in \Lambda$; all processes are launched at the time $t=0$ throughout this subsection) or $\rho_w(x, t)$ (with $X_0 = w$), it was noticed already by Kolmogorov that 
	\begin{align*}
		(\mathcal{G}^* - \partial_t)\rho(x,t) = 0 \qquad 
		\text{and}
		\qquad
		(\mathcal{G}_w - \partial_t)\rho_w(x,t) = 0,
	\end{align*}
	where the subscript in $\mathcal{G}_w$ indicates that the operator is taken with respect to the initial point variable $w$.
	
	To indeed formally conclude such PDEs with H\"{o}rmander's theory, note first that $\mathcal{G}^*$ can be expressed in terms of the vector fields~\eqref{eq:Lie generators 1}--\eqref{eq:Lie generators 2} as
	\begin{align}
		\label{eq:explicit G*}
		\mathcal{G}^* & = \tfrac{1}{2} \sum_{q=1}^d \mathcal{U}^2_q - \mathcal{U}_0 + \sum_{q=1}^d \sum_{i=1}^n \big( \partial_i \sigma_{i,q} (x) \big) \mathcal{U}_q \\
		\nonumber
		& \qquad + \tfrac{1}{2}  \sum_{i=1}^n \sum_{j=1}^n \big(\partial_{ij} a_{i,j}(x) \big) - \sum_{i=1}^n \big(\partial_{i} b_{i}(x) \big)
		,
	\end{align}
	where the two last, order zero terms are currently unimportant.
	With the similarity of $\mathcal{G}$ and $\mathcal{G}^*$, applying H\"{o}rmander's criterion in $n+1$ (resp. $2n+1$) dimensions and treating separately the time dimension, one observes~\cite[Theorems~4.7.18 and~4.7.20]{Stroock-PDE} that if condition~\eqref{eq:Hormander criterion} is satisfied
	then the operator $(\mathcal{G} - \partial_t)$ on $\mathcal{C}^\infty_c ( \Lambda \times (0, \infty) )$
	(resp. $(\mathcal{G}_x + \mathcal{G}^*_w - 2\partial_t)$ on $\mathcal{C}^\infty_c ( \Lambda  \times \Lambda \times (0, \infty) )$), is hypoelliptic. Kolmogorov's PDEs can thus be concluded:
	
	\begin{theorem}
		\label{thm:existence of a smooth density}
		Suppose that the diffusion coefficient functions $\sigma_{i, q}$ and $b_i$ in~\eqref{eq:main process of interest} can be smoothly extended to entire $\mathbb{R}^n$ with bounded derivatives of all orders, and that given $X_0 \in \Lambda$, the process $X_t$ almost surely stays inside $ \Lambda$ forever. If condition~\eqref{eq:Hormander criterion} is satisfied,
		then for each fixed $t > 0$, $X_t$ allows a density function $\rho(\cdot,t): \Lambda \to \mathbb{R}$ with respect to the Lebesgue measure on $\Lambda$ (denoted as $\rho_w(\cdot,t)$ if $X_0 = w$), and $\rho(\cdot, \cdot)$ (resp. $\rho_\cdot (\cdot, \cdot)$) is a smooth function on $\Lambda \times (0, \infty)$ (resp. on $\Lambda \times \Lambda \times (0, \infty)$) and solves the PDE $\partial_t \rho(x, t) = \mathcal{G}^*\rho(x, t)$ (resp. also the PDE $\partial_t \rho_w(x, t) = \mathcal{G}_w \rho_w(x, t)$).
	\end{theorem}
	
	\begin{remark}
		\label{rem:different Lie algebras}
		By Theorem~\ref{thm:Hormander original}, it is not hard to conclude that also the operators $\mathcal{G}$ and $\mathcal{G}^*$ on $\Lambda$, and $(\mathcal{G}^* \pm \partial_t)$ or $(\mathcal{G} \pm \partial_t)$ on $\Lambda \times (0, \infty) $ are all hypoelliptic under condition~\eqref{eq:Hormander criterion} (see, e.g.~\cite{Stroock-PDE}).
	\end{remark}

	\subsection{A slowed-down process}
	\label{subsec:slowed-down}
	
	As discussed in the introduction, a key novelty of the present paper is that we allow processes with finite life-time and the coefficients of the SDEs~\eqref{eq:main process of interest} that diverge at the domain boundary $\partial \Lambda$.
	The content of this subsection is to introduce a  slowed-down process $\hat{X}$ which (as we shall prove) is a random time-change of $X$ but ``forced'' to never hit $\partial\Lambda$. The proof of our main result is then based on the process $\hat{X}$, and more precisely on two key properties proven in this section. The first one is that martingale observables of $X$ are preserved in the time-change, and the second one is that, assuming H\"{o}rmander's criterion for the original process $X$, the slowed-down process $\hat{X}$ has smooth particle densities $\hat{\rho}$ that satisfy Kolmogorov's PDEs.
	
	
	Let now $\Theta \subset \Lambda$ is a bounded open set. We introduce a smooth cutoff function $\vartheta: \mathbb{R}^n \to [0,1]$ which is strictly positive in $\Theta$, zero on $\Theta^c$, and one in the $\epsilon$-interior of $\Theta$ for some $\epsilon > 0$. We define the functions $\hat{\sigma}_{i, j}, \hat{b}_i: \mathbb{R}^n \to \mathbb{R}$, for $1 \leq i \leq n$ and $1 \leq j \leq d$, as zero for $x \not \in \Lambda$ and by $\hat{\sigma}_{i,j} (x) := \vartheta(x) \sigma_{i,j}(x)$ and $\hat{b}_{i,j} (x) := \vartheta(x)^2 b_{i}(x)$ for $x \in \Lambda$; accordingly, we also set $\hat{a} (x):= \hat{\sigma} (x) \hat{\sigma}^T(x)$, so $\hat{a}_{i,j} (x)= \vartheta(x)^2 a_{i,j}(x)$ for $x \in \Lambda$.
	The slowed-down process is then defined via $\hat{X}_0 = X_0$ (all processes are launched at the time $0$ throughout this subsection) and the SDEs in entire $\mathbb{R}^n$
	\begin{align}
		\label{eq:SDE def on hat-X}
		\diff \hat{X}^{i}_s &= \sum_{j=1}^d \hat{\sigma}_{i,j} (\hat{X}_s) \diff \tilde{B}^{j}_s + \hat{b}_i(\hat{X}_s) \diff s
		\\
		&
		\nonumber
		\stackrel{X_s \in \Lambda}{=}\vartheta(\hat X_s) \sum_{j=1}^d \sigma_{i,j}(\hat{X}_s) \diff \tilde{B}^{j}_s + \vartheta(\hat{X}_s)^2 b_{i}(\hat{X}_s) \diff s,
	\end{align}
	where we denote the time by $s$ and the Brownian motions by $\tilde{B}^j_s$ since $X_s$ and $\hat{X}_t$ will soon be embedded in the same probability space with a random time change $s=s(t)$ and with different Brownian motions (adapted to the different clocks).

	\subsubsection*{\textbf{$\hat{X}$ is a time-change of $X$}}
	
	We now explicate the proof of the fact that $\hat{X}$ can be embedded in the same probability space with $X$ so that it is a time-change of $X$.
	For background on random time-changes of martingales, we refer to~\cite{Jacod} or its translated summary in~\cite{Kobayashi}.
	
	Let $X_t$ be the original process in the usual filtration $\mathcal{F}_t$: 
	\begin{align}
		\label{eq:main process Doob decomposition}
		X^i_t = X^i_{0} + \sum_{j=1}^d \int_{r = 0}^t \sigma_{i,j}(X_r) \diff B^{j}_r + \int_{r = 0}^t  b_{i}(X_r) \diff r, \qquad t \geq 0.
	\end{align}
	Define a random time-change from $t \in [ 0, \tau_{\partial \Theta} )$ to a new time $s \geq 0$ by $s(t) =  \int_{r = 0}^t  \vartheta (X_r)^{-2} \diff r$. Denoting $Z_s = X_{t(s)}$, this time-change can be inverted: 
	\begin{align}
		\label{eq:def of time-change}
		t(s) =  \int_{\zeta = 0}^s  \vartheta (Z_\zeta )^{2} \diff \zeta =: \beta(s), \qquad s \geq 0 .
	\end{align}
	
	Define the new filtration $\mathcal{G}_s = \mathcal{F}_{T_s}$ where
	$$T_s = \inf \{ t \geq 0: \;  \int_{r = 0}^t  \vartheta (X_r)^{-2} \diff r < s \}$$
	is the $\mathcal{F}_t$ stopping time 
	that corresponds to a given $s \geq 0$. Now, the filtration $(\mathcal{G}_s)_{s \geq 0}$ indeed satisfies the usual conditions and a time-change of a continuous local $\mathcal{F}_t$ semimartingale such as $$Z_s := X_{\beta(s)}$$ is a continuous local $\mathcal{G}_s$ semimartingale~\cite{Jacod, Kobayashi}. We identify below the law of $Z$, based on two more properties:\footnote{
		Our time-change function is increasing and continuous and hence any process satisfies the technical condition of being \textit{in synchronization} with the time change~\cite{Jacod, Kobayashi}.}
	\begin{itemize}
		\item If $M_t$ is a continuous local $\mathcal{F}_t$-martingale, then $\hat{M}_s := M_{\beta(s)}$ is continuous local $\mathcal{G}_s$-martingale.
		\item 
		For continuous local $\mathcal{F}_t$-semimartingales $A_t, B_t$, setting $\hat{A}_s := A_{\beta(s)}$ and $\hat{B}_s := B_{\beta(s)}$, the quadratic covariation process in the new time becomes $\langle \hat{A}, \hat{B} \rangle_s = \langle A, B \rangle_{\beta(s)}$.
	\end{itemize}
	While  the next result can be extracted from \cite{Oksendal-90}, we provide a short proof for the reader's convenience.

	\begin{lemma}
		\label{lem:Xhat is a time-change of X}
		The process $Z_s := X_{\beta(s)}$ 
		has the same law, up to the hitting time of $\Theta^c$, than the solution to the It\^{o} SDEs
		\begin{align*}
			\diff Z^{i}_s  = \vartheta(Z_s) \sum_{j=1}^d \sigma_{i,j}(Z_s) \diff \tilde{B}^{j}_s + \vartheta(Z_s)^2 b_{i}(Z_s) \diff s, \qquad 1 \leq i \leq n
		\end{align*}
		for $s \geq 0$, where $\tilde{B}_s$ is a $d$-dimensional standard Brownian motion with respect to the filtration $\mathcal{G}_s$.
	\end{lemma}
	
	\begin{proof}
		The finite variation part of $X$ in~\eqref{eq:main process Doob decomposition} is readily transferred to the new time parametrization: denoting $r = \beta(\zeta)$, we have $\diff r = \vartheta (Z_\zeta)^2 \diff \zeta$ and
		\begin{align*}
			\int_{r = 0}^{\beta(s)}  b_{i}(X_r) \diff r = \int_{\zeta = 0}^{s}  b_{i}(Z_\zeta ) \vartheta (Z_\zeta)^2 \diff \zeta.
		\end{align*}
		For the continuous local martingale part of $X$ in~\eqref{eq:main process Doob decomposition}, since such parts are also simply re-parametrized, the $\mathcal{G}_s$ local  martingale part of $Z^i_s$ is 
		\begin{align*}
			& \sum_{j=1}^d M^{i,j}_{\beta(s)}, \qquad \text{where} \qquad  M^{i,j}_{\beta(s)}=\int_{r = 0}^{\beta(s)} \sigma_{i,j}(X_r) \diff B^{j}_r =: \hat{M}^{i,j}_{s}.
		\end{align*}
		The original $\mathcal{F}_t$ local martingales $M^{i,j}_t$ have the quadratic covariations
		\begin{align*}
			\langle M^{i,j}, M^{i',j'} \rangle_t = \delta_{j, j'} \int_{r = 0}^t \sigma_{i, j} (X_r) \sigma_{i', j'} (X_r) \diff r.
		\end{align*}
		Recalling that covariations of continuous local martingales are simply re-parametrized, $\hat{M}^{i,j}_{s}$ have the quadratic covariations
		\begin{align*}
			\langle \hat{M}^{i,j}, \hat{M}^{i',j'} \rangle_s &= \delta_{j, j'} \int_{r = 0}^{\beta(s)} \sigma_{i, j} (X_r) \sigma_{i', j'} (X_r) \diff r \nonumber\\
			&= \delta_{j, j'} \int_{\zeta = 0}^{s} \sigma_{i, j} (Z_\zeta) \sigma_{i', j'} (Z_\zeta) \vartheta (Z_\zeta)^2 \diff \zeta.
		\end{align*}
		In particular, setting (up to the hitting of $\Theta^c$) 
		$$
		\diff m^{i, j}_s = \frac{\mathbb{I} \{ \sigma_{i, j} (Z_s) \neq 0 \}}{\sigma_{i, j} (Z_s) \vartheta (Z_s)}\diff \hat{M}^{i,j}_s,
		$$ 
		It\^{o} calculus gives 
		\begin{align*}
			\langle m^{i,j}, m^{i',j'} \rangle_s = \delta_{j, j'} \mathbb{I} \{ \sigma_{i, j} (Z_s) \neq 0 \} ds.
		\end{align*}
		and we conclude that $\diff m^{i,j}_s =  \mathbb{I} \{ \sigma_{i, j} (Z_s) \neq 0 \} \diff \tilde{B}^{j}_s$, where $\tilde{B}^j_s$, $1 \leq j \leq d$ are independent $\mathcal{G}_s$-Brownian motions\footnote{To be very precise, we might need to extend the probability space to define $\tilde{B}^j_s$ if there exist open time intervals when $\sigma_{i, j} (Z_s) = 0$.} , and 
		\begin{align*}
			\diff \text{[local martingale part of $Z^i_s$]} = \vartheta (Z_s) \sum_{j=1}^d \sigma_{i, j} (Z_s)  \diff \tilde{B}^j_s.
		\end{align*}
		Putting together the finite variation and local martingale parts, we conclude the claimed form of $Z_s$.

	\end{proof}

	
	\subsubsection*{\textbf{Martingale observables are preserved}}
	
	As in Theorem~\ref{thm:mgale observables solve PDEs}, fix $T>0$ and let $ g: \Lambda  \to \mathbb{R}$ and $ h: \Lambda \times [0,T) \to \mathbb{R}$ continuous functions.
	Due to the previous result, we embed $\hat{X}$  in the same probability space as $X$  by setting $\hat{X}_s = Z_s = X_{\beta(s)}$; note that the filtration to study $\hat{X}_s$ in is $\mathcal{G}_s$, and both processes are launched at the time $t=s=0$ in the set $\Theta$.

	Set
	\begin{align}
		\label{eq:def of time-changed gamma}
		\hat{ \gamma}_s =: \gamma_{ \beta(s)} = \exp \Big(\int_{r =0}^{\beta(s)} g(X_r )  \diff r \Big) = \exp \Big(\int_{\zeta=0}^{s} g(\hat{X}_\zeta ) \vartheta( \hat{X}_\zeta )^2  \diff \zeta \Big)
	\end{align}
	and
	\begin{align}
		\label{eq:def of time-changed H}
		\hat{H}_{ s} =:  H_{ \beta(s)}= \int_{r = 0}^{\beta(s)} \gamma_{r} h (X_r, r)  \diff r  
		= \int_{ \zeta = 0}^{s} \hat{\gamma}_{\zeta} h (\hat{X}_\zeta, \beta(\zeta))  \vartheta( \hat{X}_\zeta )^2  \diff \zeta . 
	\end{align}
	Hence for ${M}_t = f ({X}_t, t) {\gamma}_t +  H_t$  we have
	$$M_{\beta(s)} = f (\hat{X}_s, \beta(s)) \hat{\gamma}_s + \hat{ H}_s =: \hat{M}_s,$$
	which makes sense at least for $s \in [0, T)$ (since $\beta(s) \leq s$).
	
	\begin{lemma}
		\label{lem:martingale time change}
		Let $ g$ and $ h$ be continuous and $f: \Lambda \times[0,T) \to \mathbb{R}$ locally bounded and Borel measurable.
		If $f$ is a $(g, h)$-martingale observable under $X$ for $t \in [0, T)$
		then $\hat{M}_s = f (\hat{X}_s, \beta(s)) \hat{\gamma}_s + \hat{ H}_s$ is a $\mathcal{G}_s$-martingale for $s \in [0, T)$.
	\end{lemma}

	\begin{proof}
			Clearly $\hat{M}_s$ is adapted. It is also bounded on bounded time intervals (since the slowing-down restricts $\hat{X}$ onto $\Theta$) and as such integrable and a genuine martingale if it is a local martingale. For the latter, we show the conditional expectation condition with localizing stopping times.
		Let $\tau_n$ be a localizing sequence of bounded $\mathcal{F}_t$-stopping times for $M_t$ and let $s(\tau_n)$ be the corresponding $\mathcal{G}_s$-stopping time. 
		For $s'>s$ we have
		\begin{align*}
			\mathbb{E}[\hat{M}_{s'}^{s(\tau_n)} \; | \; \mathcal{G}_s] = \mathbb{E}[M^{\tau_n}_{T_{s'}} \; | \; \mathcal{F}_{T_s}] = M^{\tau_n}_{T_{s'} \wedge T_s} = \hat{M}^{s(\tau_n)}_{s' \wedge s} = \hat{M}^{s(\tau_n)}_{s},
		\end{align*}
		where we have used optional sampling theorem.
		Thus $\hat{M}$ is a local martingale (since $s(\tau_n) \geq \tau_n$ is a localizing sequence). 
	\end{proof}
	
	
	\subsubsection*{\textbf{H\"{o}rmander's criterion is preserved}}
	
	Let us first explicate the relevant operators for the slowed-down process:
	\begin{align}
		\label{eq:Uhat}
		\hat{\mathcal{U}}_q &:= \sum_{i=1}^n \hat{\sigma}_{i, q} (x) \partial_i = \vartheta(x) \mathcal{U}_q, \qquad q=1, \ldots, d  \qquad \text{and}\\
		\nonumber
		\hat{\mathcal{U}}_0 &:= \sum_{i= 1}^n \hat{b}_i(x) \partial_i - \tfrac{1}{2} \sum_{q= 1}^d \sum_{i= 1}^n \big( \hat{\mathcal{U}}_q \hat{\sigma}_{i, q} (x) \big) \partial_i \\
		\nonumber
		&= \vartheta(x)^2 \mathcal{U}_0  - \tfrac{1}{2} \sum_{q= 1}^d \sum_{i= 1}^n \sum_{j= 1}^n \big(  \vartheta(x) \sigma_{i, q} (x) \sigma_{j, q} (x) \partial_j \vartheta(x) \big) \partial_i \\
		\label{eq:Vhat}
		&= \vartheta(x)^2 \mathcal{U}_0 - \tfrac{\vartheta(x)}{2} \sum_{q= 1}^d \sum_{j= 1}^n \big(   \sigma_{j, q} (x) \partial_j \vartheta(x) \big) \mathcal{U}_q. 
	\end{align}
	
	Based on Equations~\eqref{eq:Uhat}--\eqref{eq:Vhat}, proving the statement inductively for the linear spans of more and more nested Lie brackets, one readily obtains:
	
	\begin{lemma}
		At every $x \in \Theta$, the dimensions of the two Lie algebras, generated by $\{ \mathcal{U}_i \}_{i=1}^q$ and $\{ [\mathcal{U}_i, \mathcal{U}_0] \}_{i=1}^q$, or by $\{ \hat{\mathcal{U}}_i \}_{i=1}^q$ and $\{ [\hat{\mathcal{U}}_i, \hat{\mathcal{U}}_0] \}_{i=1}^q$, coincide.
	\end{lemma}

	\subsubsection*{\textbf{Smooth particle densities and Kolmogorov's PDEs}}
	
	
	\begin{lemma}
		\label{lemma:no-hitting}
		Let $\hat{X}_0 = w \in \Theta$, and suppose that the Lie algebra generated by $\{ \mathcal{U}_q \}_{q=1}^d$ and $\{ [\mathcal{U}_0, \mathcal{U}_q] \}_{q=1}^d$ is of dimension $n$ at the point $w$. Then, the hitting time to $\Theta^c$ by $\hat{X}_s$ is almost surely infinite.
	\end{lemma}
	
	\begin{proof}
		The Hörmander criterion within Malliavin calculus~\cite[Theorem~2.3.2]{Nualart} asserts that under the Lie algebra dimension condition above (which, by the previous lemma, coincides for the processes $X_t$ and $\hat{X}_s$), $\hat{X}_s$ has a smooth density function $\hat{\rho}_{x_0}(\cdot, s) =: \hat{\rho} (\cdot, s): \mathbb{R}^n \to \mathbb{R}$, for each fixed $s>0$.\footnote{Note that the Lie algebra dimension criterion in~\cite{Nualart} is only verified at one point (not all of $\mathbb{R}^n$), a weaker assumption than, e.g., that of Theorem~\ref{thm:existence of a smooth density}. The price to pay is that~\cite[Theorem~2.3.2]{Nualart} cannot guarantee smoothness of the density in the time variable. We also remark that the generators in the statement in~\cite{Nualart} differ from those in Lemma~\ref{lemma:no-hitting}; the generated Lie algebras nevertheless coincide, by the Jacobi identity.
		} The proof is now based on a contradiction. Suppose $\tau_{\Theta^c} < \infty$. Then, by the very definition of $\hat{X}$, we have $\hat{X}_s = \hat{X}_{\tau_{\Theta^c}} \in  \Theta^c$ for all $s \geq \tau_{\Theta^c}$. Hence the density function $\hat{\rho}(\cdot,s)$ for $s\geq \tau_{\Theta^c}$ is supported on $\overline{\Theta}$ and, by continuity, equals to zero on $\partial \Theta$. On the other hand, for any fixed $s > 0$, one has $\mathbb{P}[\tau_{\Theta^c} \leq s] = \mathbb{P}[\hat{X}_s \in \Theta^c ] = \int_{x \in \Theta^c} \hat{\rho}(x, s) \diff^n x = 0$. Then it follows that $\mathbb{P}[\tau_{\Theta^c} < \infty]  = \mathbb{P}[\cup_{s \in \mathbb{N}} \{ \tau_{\Theta^c} \leq s\} ] = 0 $ leading to the contradiction. This completes the proof.
	\end{proof}

	Combining the previous lemma with Theorem~\ref{thm:existence of a smooth density}, we obtain: 
	\begin{corollary}
		\label{cor:strong Kolmogorov forward PDE for the slowed-down process}
		If the Lie algebra generated by $\{ \mathcal{U}_q \}_{q=1}^d$ and $\{ [\mathcal{U}_0, \mathcal{U}_q] \}_{q=1}^d$ is of dimension $n$ at every point $x \in \Theta$, and if $\mathbb{P}[\hat{X}_0 \in \Theta] = 1$, then for each fixed $t$, $\hat{X}_t$ (resp. $\hat{X}_t$ given $\hat{X}_0 = w \in \Theta$) has a density function $\hat{\rho}(\cdot,t): \Theta \to \mathbb{R}$ (resp. $\hat{\rho}_w (\cdot , t): \Theta \to \mathbb{R}$) with respect to the Lebesgue measure on $\Theta$, with zero boundary values at $\partial \Theta$. This density function is jointly smooth in its variables and solves the forward PDE $\partial_t \hat{\rho} (x, t) = \hat{\mathcal{G}}^*_x \hat{\rho} (x, t)$ (resp. also the backward PDE $\partial_t \hat{\rho}_w (x, t)= \hat{\mathcal{G}}_w \hat{\rho}_w (x, t)$). (Here $\hat{\mathcal{G}}$ and $\hat{\mathcal{G}}^*$ correspond to $\mathcal{G}$ and $\mathcal{G}^*$, but with $a$ and $b$ replaced by $\hat{a}$ and $\hat{b}$, respectively.)
	\end{corollary}


	\subsection{Proof of Theorem \ref{thm:mgale observables solve PDEs}}
	\label{subsec:main-proof1}
	To explain the idea, consider first the special case when $g=h = 0$ and $f(x, t) = f(x)$ is time-independent. Then changing into the slowed-down process $\hat{X}$, we would have, by Lemma~\ref{lem:martingale time change} and Corollary~\ref{cor:strong Kolmogorov forward PDE for the slowed-down process}
	\begin{align}
		\label{eq:f_simple-case}
		f(w) &= \mathbb{E}_w [f(\hat{X}_t)] = \int_{x \in \Theta} f(x) \hat{\rho}_w(x, t) \diff^n x.
	\end{align}
	This expression is smooth in $w$ since $\hat{\rho}$ is smooth by Corollary~\ref{cor:strong Kolmogorov forward PDE for the slowed-down process} again. The martingaleness of $f(\hat{X}_t)$ then leads to the PDE 
	$\mathcal{G}f = 0$. Note that the application of conformally \textit{invariant} SLE observables (Equation~\ref{eq:def of conf cov SLE observable} with $\Delta_i = 0$ for all $2 \leq i \leq n$) is handled by this simple special case.\footnote{
		Even covariant observables can be handled in this manner by considering a process $X$ on $\mathbb{R}^{2n-1}$ for which $X^i_t$, $1 \leq i \leq n$ are given by~\eqref{eq:SLE SDE} and $X^{n-1+i}_t := g_t'(X^i_t)$ for $2 \leq i \leq n$, and proving the H\"{o}rmander criterion for that process, cf.~\cite{AHPY}.
	} In a more general setting, we need to take into account non-trivial (path-dependent) $\gamma$ and $H$, as well as the time-dependence of $f$. 
	
	\begin{proof}[Proof of Theorem \ref{thm:mgale observables solve PDEs}]

		We divide the proof into steps. 
		
		\textbf{Step 1: mollify $f$ in time while preserving the martingale.}
		First, let $\psi: \mathbb{R} \to \mathbb{R}_{\geq 0}$ be a smooth density function supported on $[0,1]$ and let $\psi_\epsilon (t) = \psi (t/\epsilon)/\epsilon$. Then set for $(x, t) \in \Lambda \times [0, T-\epsilon)$
		\begin{align}
			\label{eq:f_eps}
			f_\epsilon(x, t) := \int_{s \in \mathbb{R} } f(x, s) \psi_\epsilon (s-t) \diff s = \int_{s \in \mathbb{R} } f(x, t+s) \psi_\epsilon (s) \diff s.
		\end{align}
		By the middle expression, for any fixed $x$, $f_\epsilon(x, \cdot) $ is a smooth function in the time variable with time derivatives of all orders being locally bounded in $\Lambda \times [0, T-\epsilon)$. On the other hand, the last expression can be seen as adding random delay with density $ \psi_\epsilon$ to the time when $f(X_t, t)$ is launched (due to the time-homogeneity of the It\^{o} SDE for $X_t$). Since $M_t$ in~\eqref{eq:def of mgale observable} was assumed a local martingale for any launching point and time, we conclude that also
		\begin{align*}
			M^\epsilon_t = \gamma_{t_0, t} f_\epsilon (X_t, t) +  H^\epsilon_{t_0, t}
		\end{align*}
		is a martingale for $t \in [t_0, T - \epsilon]$, where $$H^\epsilon_{t_0, t} = \int_{s= t_0}^t \gamma_{t_0, s} h_\epsilon (X_s, s) \diff s
		\quad \text{where} \quad h_\epsilon(x, t) := \int_{s \in \mathbb{R} } h(x, s) \psi_\epsilon (s-t) \diff s;$$
		indeed, this is just an average over the random launching delay of the original local martingales.
		
		For technical reasons, we shall however always launch $X_t$ at the time $t=0$ and instead add an offset $t_0$ to the arguments of $f_\epsilon$ and $h_\epsilon$. In other words, we define a collection of local martingales, $M^{\epsilon, t_0}$ indexed by the mollification parameter $\epsilon$ and a time offset $t_0$, by
		\begin{align*}
			& M^{\epsilon, t_0}_t = \gamma_{ t} f_\epsilon (X_t, t + t_0) +  H^{\epsilon, t_0}_{ t}, \qquad t \in [0, T-\epsilon - t_0],\\
			& \text{where} \qquad H^{\epsilon, t_0}_{ t} :=
			\int_{s= 0}^t \gamma_{s} h_\epsilon (X_s, s+t_0) \diff s.
		\end{align*}
		
		
		\textbf{Step 2: translate to the slowed-down process.} 
		Recall from Lemma~\ref{lem:martingale time change} that now
		$$\hat{M}_s := M^{\epsilon, t_0}_{t = \beta(s)} = f_\epsilon (\hat{X}_s, \beta(s) + t_0) \hat{\gamma}_s + \hat{ H}^{t_0}_s, \qquad 0 \leq s < T- \epsilon - t_0,$$
		is a martingale (in the filtration $\mathcal{G}_s$, which however will not play any r\^{o}le in this proof). Here $\hat{X}$ is defined in~\eqref{eq:SDE def on hat-X}, $\beta(s)$ in~\eqref{eq:def of time-change}, $\hat{\gamma}_s$ in~\eqref{eq:def of time-changed gamma} and 
		\begin{align}
			\label{eq:def of time-changed H 2}
			\hat{H}^{t_0}_{ s} &=: H^{\epsilon, t_0}_{ t=\beta(s)}= \int_{r = 0}^{\beta(s)} \gamma_{r} h_\epsilon (X_r, r+t_0)  \diff r \\ 
			\nonumber
			& = \int_{ \zeta = 0}^{s} \hat{\gamma}_{\zeta} h_\epsilon (\hat{X}_\zeta, \beta(\zeta)+t_0)  \vartheta( \hat{X}_\zeta )^2  \diff \zeta . 
		\end{align}

		\textbf{Step 3: continuity of $f_\epsilon$.} 
		
		\begin{lemma}
			\label{lem:f_eps is continuous}
			The function $f_\epsilon:\Lambda \times [0, T-\epsilon] \to \mathbb{R}$ defined by~\eqref{eq:f_eps} is continuous. 
		\end{lemma}
		
		The proof of the lemma is postponed into Section~\ref{subsec:technical stuff}.

		\begin{remark}
			\label{remark:continuity}
			In the time-invariant case $f(x, t) = f(x)$, we have $f_\epsilon = f$, and as such $f$ is continuous, cf.~\eqref{eq:f_simple-case}.
		\end{remark}
		
		\textbf{Step 4: encoding test functions as random launching time and point.} The content of steps 4--6 is to prove that $f_\epsilon$ is a weak solution to the almost the same PDE on the time interval $(0, T-\epsilon)$. Namely, for any smooth and compactly-supported $\varphi: \Lambda \times (0, T-\epsilon) \to \mathbb{R}$, one has
		\begin{align}
			\int_{v \in \Lambda} \int_{t = 0}^{T- \epsilon} f_\epsilon (v, t) \left( \mathcal{G}_v^* - \partial_{t} + g(v) \right) \varphi (v, t) + h_\epsilon(v, t)  \varphi (v, t)  \diff t \diff^n v = 0.\label{eq:weak PDE with epsilons}
		\end{align}
		
		For the proof,
		without loss of generality, we may assume here that $\varphi$ is a probability density function.
		Then, we let $(w, t_0) \in \Lambda \times [0, T-\epsilon]$ be a pair of random launching point and time offset with the density $\varphi$. We assume that $(w, t_0)$ is $\mathcal{F}_0$-measurable and we generate $X$ via~\eqref{eq:main process of interest} and Brownian motions independent of $(w, t_0)$.  We also assume that the slowing-down function $\vartheta$ in~\eqref{eq:SDE def on hat-X} is such that the supports of $\varphi(\cdot, t)$ belong to the $\delta$-interior of the preimage $\vartheta^{-1}(1)$, with the same $\delta > 0$ for all $t$, and we define the slowed-down process $\hat{X}_s$ through it.
		
		With the martingale property of $\hat{M}^{t_0}_s$ for $0 \leq s \leq T- \epsilon - t_0$ (step 2), choose $s$ small enough (compared to the compact support of $\varphi$) to be on this interval irrespectively of the random $t_0$. Then (also using $\hat{\gamma}_{0}=1$ and $\hat{ H}^{t_0}_{0} = 0$), we obtain
		\begin{align}
			0 &= \mathbb{E} \Big[\frac{\hat{M}^{t_0}_s - \hat{M}^{t_0}_{0}}{s} \Big] \nonumber
			\\
			& = \mathbb{E} \Big[ \frac{\hat{\gamma}_{s} - \hat{\gamma}_{0}}{s} f_\epsilon (\hat{X}_{ s},  \beta( s)+t_0) \Big]\nonumber
			\\
			&+ \mathbb{E} \Big[ \frac{f_\epsilon (\hat{X}_{ s}, \beta( s) + t_0) - f_\epsilon (\hat{X}_{0}, t_0)}{s} \Big] + \mathbb{E} \Big[ \frac{ \hat{ H}^{t_0}_{s} }{s} \Big]\label{eq:three terms}.
		\end{align}
		The idea is to pass to the limit $s \downarrow 0$. The first and third term are then readily handled by the a.s. continuity of the paths $s \mapsto\hat{X}_{s}$. Indeed, as $s \to 0$, we have $(\hat{\gamma}_{s} - \hat{\gamma}_{0})/s \to g(\hat{X}_{0}) \vartheta(\hat{X}_{0})^2 = g(\hat{X}_{0}) $, and $f_\epsilon (\hat{X}_{ s }, \beta( s) + t_0) \to f_\epsilon (\hat{X}_{0},t_0)$ (by step 3), and $\hat{ H}^{t_0}_{s} / s \to \hat{\gamma}_{0} h_\epsilon (\hat{X}_{0}, t_0) \vartheta(\hat{X}_{0})^2 = h_\epsilon (\hat{X}_{0}, t_0)$, all almost surely. Since all the functions involved are continuous and we are restricted onto a compact set by the slowing-down, dominated convergence gives us
		\begin{align*}
			\mathbb{E} \Big[   \frac{\hat{\gamma}_{s} - \hat{\gamma}_{0}}{s} f_\epsilon (\hat{X}_{ s},  \beta( s)+t_0) \Big] 
			&\stackrel{s \downarrow 0 }{\longrightarrow} \mathbb{E} [  g(\hat{X}_{0}) f_\epsilon (\hat{X}_{0}, t_0) ] \nonumber \\
			&= \int_{t=0}^{T-\epsilon} \int_{v \in \Lambda} g(v) f_\epsilon (v, t) \varphi(v, t )   \diff^n v  \diff t \\
			\text{and} \quad
			\mathbb{E} \Big[ \frac{ \hat{ H}^{t_0}_{s} }{s} \Big] 
			\stackrel{s \downarrow 0 }{\longrightarrow} \mathbb{E} [  h_\epsilon (\hat{X}_{0}, t_0) ] & = \int_{t=0}^{T-\epsilon} \int_{v \in \Lambda} h_\epsilon (v, t)  \varphi(v, t ) \diff^n v \diff t.
		\end{align*}
		As such, it remains to analyze the second term in~\eqref{eq:three terms}. 
		
		\textbf{Step 5: Removing the time change at small times $s$.}
		The following lemma shows that we can replace the path-dependent function $\beta(s)$ by $s$ in the second term in~\eqref{eq:three terms}. The proof is postponed to Section~\ref{subsec:technical stuff}.

		
		\begin{lemma}
			\label{lem:remove time-change}
			For any fixed compactly-supported launching-point density $\varphi$, we have 
			\begin{align*}
				\mathbb{E}[ f_\epsilon(\hat{X}_{s},\beta( s)+t_0)] =  \mathbb{E}[ f_\epsilon(\hat{X}_{s}, s+t_0)] + o(s), \quad \text{as } s \to 0.
			\end{align*}
		\end{lemma}
		
		In particular, the second term of~\eqref{eq:three terms} becomes, as $s \downarrow 0$,
		\begin{align*}
			\mathbb{E} \Big[ \frac{f_\epsilon (\hat{X}_{ s}, \beta( s)+t_0) - f_\epsilon (\hat{X}_{0}, t_0)}{s} \Big] = \mathbb{E} \Big[ \frac{f_\epsilon (\hat{X}_{ s }, t_0 + s) - f_\epsilon (\hat{X}_{0}, t_0)}{s} \Big] + o(1).
		\end{align*}
		
		\textbf{Step 6: explicit derivations.} Now that the path-dependence has been removed in Step 5, we will analyze the right-derivative of the expectation $E[ f_\epsilon(\hat{X}_{s},t_0+s)]$ at $s=0$ via the particle density $\hat{\rho}$. By the Intermediate value theorem, if a continuous function $\mathbb{R}_{\geq 0} \to 0$ (in this case $s \mapsto E[ f_\epsilon(\hat{X}_{s},t_0+s)]$) is differentiable on $\mathbb{R}_{>0}$ and its derivative has a right limit at $0$, then the function is right-differentiable with continuous right derivative at $0$. Hence, we fix for a moment $s>0$ and study (two-sided) derivatives.
		We have
		\begin{align*}
			\mathbb{E}[f_\epsilon (\hat{X}_{s}, t_0 + s)] 
			&= \int_{x \in \Lambda} \int_{t=0}^{T-\epsilon} \int_{v \in \Lambda} \big( \hat\rho_v(x,s)\varphi(v,t)f_\epsilon(x,t+s) \big) \diff^n v \diff t \diff^n x.
		\end{align*}
		now we claim that $\partial_s$ can be taken inside the integral. For this, note that the support of $\hat\rho_w(x,s)\varphi(w,t)$ is (for all $s > 0$) contained in $(x, w, t) \subset \overline{\Theta} \times supp(\varphi)$ and the time derivatives of $f_\epsilon$ are locally bounded, while $\hat\rho$ is smooth. Hence,  Leibniz' criterion for differentiation under integral applies, and we get
		\begin{align*}
			\partial_s  \mathbb{E}[f_\epsilon (\hat{X}_{s}, t + s)] 
			&=  \int_{x \in \Lambda} \int_{t=0}^{T-\epsilon} \int_{v \in \Lambda} \partial_s \hat\rho_v(x,s)\varphi(v,t)f_\epsilon(x,t+s) \\
			& \qquad + \hat\rho_v(x,s)\varphi(v,t)\partial_s f_\epsilon(x,t+s) \diff^n v \diff t \diff^n x \\
			&= \int_{x \in \Lambda} \int_{t=0}^{T-\epsilon} \int_{v \in \Lambda}  \mathcal{G}_v \hat\rho_v(x,s)\varphi(v,t)f_\epsilon(x,t+s)\\
			&\qquad + \hat\rho_v(x,s) \varphi(v,t) \partial_t f_\epsilon(x,t+s)  \diff^n v \diff t \diff^n x.
		\end{align*}
		Above we used Kolmogorov's backward equation and replaced $\hat{\mathcal{G}}_v$ by ${\mathcal{G}}_v$ since $\varphi(v,t)$ is supported on the $\epsilon$ interior of $\Theta$, where $\hat{\mathcal{G}}$ coincides with $\mathcal{G}$.
		Next integration by parts in $v$-variable for the first term and $t$ for the second gives
		\begin{align*}
			&\partial_s  \mathbb{E}[f_\epsilon (\hat{X}_{s}, t + s)]\\
			&=
			\int_{v \in \Lambda} \int_{t=0}^{T-\epsilon} \big( ( \mathcal{G}_v^*  - \partial_t )\varphi(v,t) \big) \Big(  \int_{x \in \Lambda}   \hat\rho_v(x,s) f_\epsilon(x,t+s)  \diff^n x \Big) \diff t  \diff^n v 
			\\
			&=
			\int_{v \in \Lambda} \int_{t=0}^{T-\epsilon} \big( ( \mathcal{G}_v^*  - \partial_t )\varphi(v,t) \big) \mathbb{E}_{v} [f_\epsilon (\hat{X}_{s}, t+s)] \diff t  \diff^n v. 
		\end{align*}
		This is our expression for the $s$-derivative.
		
		Letting  now $s \downarrow 0$, the expectation tends to $f_\epsilon(v,t) $ since $\hat{X}_{s} \to \hat{X}_0$ weakly and $f_\epsilon$ is continuous by step 3. The limit $s \downarrow 0$ can be taken inside the integral due to Dominated convergence as $\varphi$ is compactly-supported and $f_\epsilon$ is bounded. We obtain
		\begin{align*}
			\partial_s  \mathbb{E}[f_\epsilon (\hat{X}_{s}, t + s)]  \stackrel{s \to 0}{ \longrightarrow } \int_{v \in \Lambda} \int_{t=0}^{T-\epsilon} f_\epsilon(v,t)\big( \mathcal{G}^*_v -\partial_t \big) \varphi(v,t) \diff t  \diff^n v.
		\end{align*}
		
		\textbf{Step 7: Conclusion.}
		Steps 4--6 lead to~\eqref{eq:weak PDE with epsilons}.
		A simple convolution argument shows that for any continuous $\chi$ we have
		\begin{align*}
			\int_{t= 0}^{T-\epsilon} f_\epsilon(x,t) \chi (t) \diff t \to \int_{t= 0}^{T} f(x,t) \chi (t) dt 
		\end{align*}
		as $\epsilon \to 0$, and similarly for $h_\epsilon$ and $h$. The limit $\epsilon \to 0$ can be taken inside the spatial integral in~\eqref{eq:weak PDE with epsilons} by dominated convergence and the fact that we are dealing with compactly supported functions. Thus passing to the limit $\epsilon\to0$ in~\eqref{eq:weak PDE with epsilons}, we obtain~\eqref{eq:weak main PDE} and conclude the proof of Theorem~\ref{thm:mgale observables solve PDEs}.
	\end{proof}

	\subsection{Proofs of Theorems \ref{thm:Feynman--Kac} and \ref{thm:X-harmonic FK}}
	\label{subsec:main-proof2}
	The strategies of both proofs are the same: we check that the solution candidate $f$ given in~\eqref{eq: Feynman--Kac solution formula} or~\eqref{eq:X-harmonic FK formula} is a martingale observable. Then it is a PDE solution by Theorem~\ref{thm:mgale observables solve PDEs}, and as such it suffices to verify the boundary conditions. The first fact is recorded in the following lemma. 
	
	\begin{lemma}
		\label{lem:Feynman--Kac mgale obsesrvable}
		Assume that $f$ given in~\eqref{eq: Feynman--Kac solution formula} (resp.~\eqref{eq:X-harmonic FK formula}) is well-defined. Then, given any $t_0$ and $X_{t_0} = x$, the process $M_t$ defined in~\eqref{eq:def of mgale observable}
		is a martingale up to the stopping time $\tau = \tau_{\partial \Lambda} \wedge T$ (up to $\tau = \tau_{\partial \Lambda} $, respectively).
	\end{lemma}
	
	\begin{proof}
		The process~\eqref{eq:def of mgale observable} is clearly adapted. To check integrability and the conditional expectation condition, note first that $\gamma_{\tau} = \gamma_{t}\gamma_{t, \tau}$ and $H_{\tau} = H_t + \gamma_t H_{t, \tau}$. We now compute the conditional expectation of the random variable $Y:= \gamma_{\tau} \psi(X_\tau, \tau) + H_{\tau}$ whose expected value defines $f(x,0)$ (and hence $Y$ is integrable by assumption). We obtain
		\begin{align*}
			\mathbb{E} [ Y \; |\; \mathcal{F}_t]
			&= \mathbb{E} [ \gamma_{\tau} \psi(X_\tau, \tau) + H_{\tau} \; |\; \mathcal{F}_t]
			= \mathbb{E} [\gamma_t \gamma_{t,\tau} \psi(X_\tau, \tau) + H_t + \gamma_t H_{t,\tau} \; |\; \mathcal{F}_t] \\
			&= \gamma_t \mathbb{E} [ \gamma_{t,\tau} \psi(X_\tau, \tau) + H_{t,\tau} \; |\; \mathcal{F}_t] + H_t = M_t.
		\end{align*}
		Hence, being a conditional expectation process, $M_t$ is a martingale.
	\end{proof}
	
	Regarding the boundary values, we record the following two results whose proofs are postponed to Section~\ref{subsec:technical stuff}. In both lemmas below, we assume that the solution candidate $f$ given in~\eqref{eq: Feynman--Kac solution formula} or~\eqref{eq:X-harmonic FK formula} is well-defined, $\Lambda$ has $X$-regular boundary, $\psi, h$, and $g$ are  bounded, and all three functions are continuous in their domains of definition.
	
	\begin{lemma}
		\label{lem:bdary values of FK observables}
		Under the above assumptions,
		the Feynman--Kac solution candidate $f$ in~\eqref{eq: Feynman--Kac solution formula} has the boundary limits $\lim_{(w, s ) \to (x, t)} f(w, s) = \psi (x, t) $ for all $(x, t)$ on the cylinder boundary. 
	\end{lemma}
	
	\begin{lemma}
		\label{lem:bdary values of X-harmonic observables}
		Under the above assumptions, if additionally either 
		\begin{itemize}[noitemsep]
			\item [i)] $C = \sup_{y \in \Lambda} g(y) \leq 0$; or
			\item[ii)] the following expectations are finite for all $w \in \Lambda$ and converge: $\mathbb{E}_w [e^{C \tau}] \to 1$ as $w \to x \in \partial \Lambda$,
		\end{itemize} 
		then the $X$-harmonic solution candidate $f$ in~\eqref{eq:X-harmonic FK formula} has the boundary limits $\lim_{w \to x} f(w) = \psi (x) $.
	\end{lemma}
	
	\begin{proof}[Proof of Theorem \ref{thm:Feynman--Kac}]
		The criteria of Lemma~\ref{lem:Feynman--Kac mgale obsesrvable}, Theorem~\ref{thm:mgale observables solve PDEs} and Lemma~\ref{lem:bdary values of FK observables} are easily verified so~\eqref{eq: Feynman--Kac solution formula} is a weak solution to~\eqref{eq:parabolic BVP}. By hypoellipticity (Theorem~\ref{thm:Hormander original} and Remark~\ref{rem:different Lie algebras}), the weak solution is smooth provided that $g$ and $h$ are smooth.
		%
		%
		In the latter case, if there was another bounded solution $\tilde{f}$, then a direct computation shows that $\gamma_{a, t} \tilde{f}(X_t, t) + H_{a, t}$, given any launching time $a$ and point $X_a  = x$, is a continuous local martingale and bounded, and thus also a genuine martingale. By optional stopping, we then have
		\begin{align*}
			\tilde{f}(x, t) = \mathbb{E} [\gamma_{t,\tau} \psi(X_\tau, \tau) + H_{t,\tau} \; |\; X_t = x]
			= f(x, t).
		\end{align*}
		This concludes the proof.
	\end{proof}
	
	\begin{proof}[Proof of Theorem \ref{thm:X-harmonic FK}]
		By a variant of the computations in the proof of Lemma~\ref{lem:bdary values of X-harmonic observables} and using the fact that either $\mathbb{E}_w [e^{C \tau}] $ is finite or $C=h=0$ and $\tau <\infty$, $f$ is well-defined. 
		Now the criteria of Lemma~\ref{lem:Feynman--Kac mgale obsesrvable}, Theorem~\ref{thm:mgale observables solve PDEs} and Lemma~\ref{lem:bdary values of X-harmonic observables} are again easily verified so~\eqref{eq: Feynman--Kac solution formula} is a weak solution to~\eqref{eq:parabolic BVP}. The continuity follows from Remark~\ref{remark:continuity}, and the regularity is  argued by hypoellipticity as above for smooth $g$ and $h$.
		%
		Again, if there was another bounded strong solution $\tilde{f}$, then a direct computation shows that $\tilde{M}_t := \gamma_{ t} \tilde{f}(X_t) + H_{ t}$, given any launching point $X_0  = x$, is a continuous local martingale. In this case we have $ |\tilde{M}_t| \leq \gamma_t \Vert \tilde{f} \Vert_\infty + |H_t| $ and since the right-hand side is bounded over any compact time interval $t \in [0,T]$ (see~\eqref{eq:bound on H} below), we observe that $\tilde{M}_t $ is a genuine martingale. It remains to apply optional stopping to conclude
		\begin{align*}
			\tilde{f}(x) = \mathbb{E}_x [\gamma_{\tau} \psi(X_\tau) + H_\tau ]
			= f(x).
		\end{align*}
		With assumption (b), $\gamma_t \leq 1$ and $H = 0$, so the martingale $\gamma_t \tilde{f}(X_t)$ is bounded and optional stopping applies.
	
			With assumption (c) (which also holds under (a)), one needs to show the general conditions of optional stopping: $\tau < \infty$ a.s. (which was assumed in the statement), $\tilde{M}_\tau$ is integrable, and $\mathbb{E}_x [|\tilde{M}_T| \mathbb{I} \{\tau > T\}] \stackrel{T \to \infty}{\longrightarrow} 0$. For the latter two, note that $ |\tilde{M}_t| \leq \gamma_t \Vert \tilde{f} \Vert_\infty + |H_t| $ implies $|\tilde{M}_\tau| \leq K e^{C \tau}$ (see again~\eqref{eq:bound on H}). Secondly, by Markov's inequality, $\mathbb{P}_x[\tau > T] \leq e^{-CT} \mathbb{E}_x[\tau^\alpha e^{C \tau}]/T^\alpha$ so
		\begin{align*}
			\mathbb{E}_x [|M_T| \mathbb{I} \{\tau > T\}] \leq Ke^{CT}e^{-CT} \mathbb{E}_x[\tau^\alpha e^{C \tau}]/T^\alpha \to 0.
		\end{align*}
		This concludes the proof.
	\end{proof}

	\subsection{Proofs of the technical lemmas}
	\label{subsec:technical stuff}
	
	\begin{proof}[Proof of Lemma~\ref{lem:f_eps is continuous}]
		Fix $(x, t) \in \Lambda \times [0, T-\epsilon]$ and $\varepsilon > 0$; we shall show that for all $(x', t')$ close enough to $(x, t)$, $| f_\epsilon(x, t) - f_\epsilon(x', t') | < \varepsilon$. For several compactness arguments below, we fix throughout the proof an $r > 0$ such that $(t-r, t+2r) \subset [0, T-\epsilon] $ and $B(x, 4r) \subset \Lambda$, and we require $|t-t'|<r$ and $|x-x'|<r$ throughout. First, we have
		\begin{align*}
			| f_\epsilon(x, t) - f_\epsilon(x', t') |
			& \leq | f_\epsilon(x, t) - f_\epsilon(x, t') |
			+ | f_\epsilon(x, t') - f_\epsilon(x', t') | 
			\\
			& < C_1 |t-t'| + | f_\epsilon(x, t') - f_\epsilon(x', t') |,
		\end{align*}
		where we used the local boundedness of the time derivatives $\partial_s f_\epsilon (x, s)$ for fixed $x$ (we could denote $C_1 =C_1 (r, x)$, but we omit fixed quantities from such constants throughout this proof).
		The second term can be handled by using the martingale property. To this end, we will use the slowed-down process $\hat{X}$ with the slowing-down function $\vartheta$ such that $\vartheta(y)=1$ for all $|y-x| < 4r$ and introduce an additional time increment parameter $s \in (0, r)$. Let now $w \in \{x, x'\}$; launching $\hat{X}$ from $\hat{X}_{0}=w$ and studying the martingale with the time offset parameter $t_0 = t'$, we have
		\begin{align*}
			f_\epsilon(w, t') &= \hat{M}^{t'}_0 = \mathbb{E}_{w} [\hat{\gamma}_s f_\epsilon(\hat{X}_{s}, \beta(s)+t') + \hat{H}^{t'}_s].
		\end{align*}
		The strategy will be to approximate the above expectation for small $s$ with the path-independent quantity $\mathbb{E}_{w}  [ f_\epsilon(\hat{X}_{s}, t' + s) ]$ and analyze this through the particle density $\hat{\rho}_w(x, s)$. To this end, let $E(s, 4r)$ denote the event that $\hat{X}_s$ remains in $B(x, 4r)$ up to the time $s$ (note that then $\beta(s)= s$). It follows from the local boundedness of $f$, $g$ and $h$ that on the event $E(s, 4r)$
		\begin{align*}
			| \hat{\gamma}_s f_\epsilon(\hat{X}_{s}, t' + \beta(s)) + \hat{H}_s - 
			f_\epsilon(\hat{X}_{s}, t' + s) | < C_2 s,
		\end{align*}
		where $C_2 $ only depends on the input functions and $r$.
		Using again the local boundedness of $\hat{\gamma}$, $f$ and $h_\epsilon$
		\begin{align*}
			\big|  \mathbb{E}_{w} [\hat{\gamma}_s f_\epsilon(\hat{X}_{s}, \beta(s)+t') + \hat{H}^{t'}_s] -  \mathbb{E}_{w} [f_\epsilon(\hat{X}_{s}, t'+s) ] \big| \leq C_2 s + C_3 \mathbb{P}_{w} [E(s, 4r)^c].
		\end{align*}
		Clearly, choosing $w=x$, we have $\mathbb{P}_{x} [E (s,4r)] \to 1$ as $s\to 0$. Moreover, by the continuity of solutions to Lipschitz SDEs in the launching point, the processes launched from $x$ and $w \in B(x, r)$ can be coupled to stay at a distance $2r$ from each other up to a small enough time $s$, with a large probability that only depends on $s$ and the Lipschitz constants. Thus, for all $w \in B(x, r)$,
		\begin{align*}
			\mathbb{P}_{w} [E(s, 4r)] \geq 1- o_s (1),
		\end{align*}
		where the Landau notation means ``$o(1)$ as $s \to 0$'' and the error term is bounded uniformly over $w \in B(x, r)$. 
		We conclude that for all $|w-x| < r$ and $s < r$
		\begin{align*}
			\big|f_\epsilon (w, t') - \mathbb{E}_{w} [f_\epsilon(\hat{X}_s, t' + s) ] \big| \leq C_2 s + o_s (1) .
		\end{align*}
		
		It now remains to estimate $\big| \mathbb{E}_{x} [f_\epsilon(\hat{X}_{s}, t'+s)] - \mathbb{E}_{x'} [f_\epsilon(\hat{X}_{s}, t'+s)] \big|$. Denote
		\begin{align*}
			\mathbb{E}_{w} [f_\epsilon(\hat{X}_{s}, t'+s)]
			=: e(w, t')= \int_{y \in \Theta} f_\epsilon(y, t'+s) \hat{\rho}_w(y,s) \diff^n y. 
		\end{align*}
		By compactness of $\overline{\Theta}$, the local boundedness of $f_\epsilon$, and the smoothness of $\hat{\rho}$, the expression above is differentiable in $w$. Moreover, for fixed $x, t, r$, with $|w-x| < r$ and $|t-t'|<r$, and for a fixed\footnote{Note that if we let $s \to 0$, $\hat{\rho}_w(y,s)$ will tend to a delta distribution and we will lose the desired control of its gradient.} $s < r$, simple local boundedness arguments give
		\begin{align*}
			| \nabla_w e(w, t') | \leq C_3 ( s).
		\end{align*}
		Consequently,
		\begin{align*}
			\big| \mathbb{E}_{x} [f_\epsilon(\hat{X}_{s}, t'+s)] - \mathbb{E}_{x'} [f_\epsilon(\hat{X}_{s}, t'+s)] \big| \leq 
			C_3 (s) |x-x'|,
		\end{align*}
		and combining all the estimates so far gives
		\begin{align*}
			| f_\epsilon(x, t) - f_\epsilon(x', t') | &< C_1 |t-t'| + 2C_2 s + o_s (1) + C_3 (s) |x-x'|.
		\end{align*}
		Fixing now first the additional parameter $s$ small enough so that $2C_2 s + o_s (1) < \varepsilon/2$, we conclude the proof: for all $|t-t'|$ and $|x-x'|$ small enough so that $C_1 |t-t'| + C_3 (s) |x-x'| < \varepsilon/2$, we have $| f_\epsilon(x, t) - f_\epsilon(x', t') | < \varepsilon$.

	\end{proof}

	\begin{proof}[Proof of Lemma~\ref{lem:remove time-change}]
		Let $A$ denote 
		preimage $\vartheta^{-1}(1)$, and thus $\beta(s)=s$ for all $s\leq \tau_{A^c} = \inf\{s\geq 0: \hat{X}_s \in A^c\}$.
		We thus decompose
		\begin{align*}
			\mathbb{E}[ f_\epsilon(\hat{X}_{s},t+\beta(s))]
			&=  \mathbb{E}[f_\epsilon(\hat{X}_{s},t+\beta(s))I_{\tau_{A^c}\leq s}] +  \mathbb{E}[ f_\epsilon(\hat X_{s},t+s)I_{\tau_{A^c}> s}].
		\end{align*}
		Similarly, we have
		\begin{align*}
			\mathbb{E}[ f_\epsilon(\hat X_{s},t+s)] 
			&=  \mathbb{E}[ f_\epsilon(\hat X_{s},t+s)I_{\tau_{A^c}\leq s}] +  \mathbb{E}[ f_\epsilon(\hat X_{s},t+s)I_{\tau_{A^c}> s}].
		\end{align*}
		Since $f_\epsilon$ is locally bounded, it hence suffices to prove that
		$$
		\mathbb{P}[\tau_{A^c}\leq s ] = o(s).
		$$
		For this, note first that since the functions $\hat{\sigma}_{i,j}(x)$ and $\hat{b}_i(x)$ are smooth and bounded, directly from the It\^{o} isometry and the Burkholder--Davis--Gundy inequality one obtains\footnote{Let us exemplify the argument for one-dimensional diffusions. Then if $\Vert \cdot\Vert_p$ denotes the $L^p(\mathbb{P})$ norm, using Minkowski's inequality gives
			$$
			\Vert X_s -X_0 \Vert_p \leq \int_0^s \Vert b(X_u)\Vert_p du + \Vert \int_0^s \sigma(X_u)dB_u\Vert_p,
			$$
			where, thanks also to Burkholder--Davis--Gundy inequality,
			$$
			\Vert\int_0^s \sigma(X_u)dB_u\Vert_p \leq C_p \Vert \int_0^s \sigma^2(X_u)du \Vert_{p/2}^{1/2} \leq C_{p,\sigma} s^{1/2},
			$$
			leading eventually to the claimed inequality.}
		$$
		\mathbb{E}\big[|\hat{X}_{s}-\hat{X}_0|^p\big] \leq C_p s^{\frac{p}{2}}
		$$
		for any $p\geq 1$, where  $C_p$ is a constant depending on $p$ and the functions $\hat{\sigma}_{i,j}$ and $\hat{b}_i$. As such, the Kolmogorov--Chentshov theorem implies that 
		we have the H\"{o}lder continuity
		$$
		|\hat{X}_{s}-\hat{X}_0| \leq C_{[0,s]}s^{\gamma}
		$$
		for all $\gamma<\frac12$, and for some random  constant $C_{[0,s]}$ (depending on $\gamma$). Moreover, the arguments, e.g., in \cite{Nummi-Viitasaari24} (see also \cite[Proof of Theorem 1]{Yazigi14}), gives, for a constant $C=C(\gamma,p)$, that
		$$
		\mathbb{E} \big[C_{[0,s]}^p\big] \leq Cs^{\frac{p}{2}}.
		$$
		Now since the slowing-down function $\vartheta$ was originally constructed so that $supp(\varphi)$ is contained in $\delta$-interior of $A$, i.e. we launch the process at time $t$ from the $\delta$-interior of $A$, this means that on the event $\tau_{A^c}\leq s$ we have 
		$$
		\delta < |\hat{X}_{\tau_{A^c}}-\hat{X}_0| \leq C_{[0,s]}|s|^{1/2-\epsilon} \leq C_{[0,s]}.
		$$
		Thus
		$$
		\mathbb{P}[\tau_{A^c}\leq s] \leq \mathbb{P}[C_{[0,s]} > \delta] \leq \tfrac{\mathbb{E} [ C_{[0,s]}^p ]}{\delta^p} \leq C\delta^{-p}s^{\frac{p}{2}}.
		$$
		Choosing $p>2$ completes the proof.
	\end{proof}
	
	\begin{remark}
		Note that the above proof shows actually that, for any $p\geq 1$, we have
		\begin{align*}
			\mathbb{E}[ f_\epsilon(\hat{X}_{s},\beta( s)+t_0)] =  \mathbb{E}[ f_\epsilon(\hat{X}_{s}, s+t_0)] + o(s^p) \qquad \text{as } s \to 0.
		\end{align*}
	\end{remark}
	
	\begin{proof}[Proof of Lemma~\ref{lem:bdary values of FK observables}]
		Suppose first that the limit point satisfies $(x, t) = (x, T) \in \Lambda \times \{ T \}$. Using
		\begin{align*}
			|f(w, s) - \psi (x, t)| \leq \mathbb{E}_{w, s} [|\gamma_{s,\tau} - 1|| \psi(X_\tau, \tau)| +  | \psi(X_\tau, \tau)  -  \psi(x, T)| + |H_{s, \tau} |],
		\end{align*}
		and the uniform boundedness of $\psi, h$, and $g$,
		this case boils down to showing 
		\begin{align}
			\label{eq:needed}
			\mathbb{E}_{w, s}[| \psi(X_\tau, \tau)  - \psi(x, T)|] \to 0 \qquad \text{as } (w, s) \to (x, T).
		\end{align}
		Due to the time-homogeneity of the process $X$, instead of taking the launching time $s$ towards the termination time $T$, we can think of starting the process at time $0$ and stopping it at latest at a very small time $(T-s)\to 0$. Basic SDE theory then allows one to conclude that $(X_\tau, \tau) \to (x, T) $ weakly, which proves \eqref{eq:needed}.
		
		Suppose then that the limit point satisfies $(x, t) \in \partial \Lambda \times [0, T]$ and fix $\epsilon > 0$.
		For starters, compute:
		\begin{align}
			\nonumber 
			|f(w, s) - \psi(x, t)| 
			&\leq
			|f(w, s) - \psi(x, s)| + |\psi(x, s) - \psi(x, t)|
			\\
			\nonumber
			& \leq \mathbb{E}_{w, s} [|\gamma_{s,\tau} - 1|\cdot | \psi(X_\tau, \tau)| +  | \psi(X_\tau, \tau)  -  \psi(x, s)| + |H_{s, \tau} |] \\
			& \qquad +  |\psi(x, s) - \psi(x, t)|.
			\label{eq:another three term sum}
		\end{align}
		For compactness arguments, let us fix an $r>0$, and assume below $|s-t|<r$ (apart from that, independence of $s$ is crucial). Set $K = \overline{B(x, r)} \cap \partial \Lambda$; by compactness, $\psi$ is uniformly continuous on $K \times [t-2r, t+ 2r]$. Thus, there is $\delta > 0$ (depending only on $\epsilon, \psi$ and $r$) such that if $\tau < s+ \delta$ and $  |X_{\tau} - x| < \delta$ then
		\begin{align}
			\label{eq:ineq 1}
			|\psi(X_\tau, \tau) - \psi(x,s)| < \epsilon/5.
		\end{align}
		Next, by the uniform boundedness of $\psi, h$ and $g$, taking $\delta $ smaller if necessary (depending on $\epsilon, \psi, g, h$), if $\tau < s+ \delta$ and $  |X_{\tau} - x| < \delta$, then we also have
		\begin{align}
			\label{eq:ineq 2}
			|\gamma_{s,\tau} - 1|\cdot | \psi(X_\tau, \tau)| < \epsilon / 5
			\qquad \text{and} \qquad 
			|H_{s, \tau} | < \epsilon / 5.
		\end{align}
		Define thus the event $E(\delta) =  \{ \tau < s+ \delta$ and $ | X_{\tau} - x| < \delta \}$, on which~\eqref{eq:ineq 1}--\eqref{eq:ineq 2} hold. On the event $E(\delta)^c$, we just use the fact that the random variable in the expectation of~\eqref{eq:another three term sum} is bounded by some $C$ (depending only on the input functions $g, h, \psi$), to conclude
		\begin{align*}
			\eqref{eq:another three term sum}
			\leq
			3 \epsilon / 5 + C \mathbb{P}_{w, s} [ E(\delta)^c ]  +  |\psi(x, s) - \psi(x, t)|.
		\end{align*}
		Finally, by the definition of a regular boundary, we see that there is $\delta_2 > 0$ (independent of $s$ due to the time-homogeneity of $X$) such that if $|w - x| < \delta_2$, then $ C \mathbb{P}_{w, s} [ E(\delta)^c ] < \epsilon/5$. By the continuity of $\psi$, taking $\delta_2$ smaller if necessary, we also have $|\psi(x, s) - \psi(x, t)| < \epsilon / 5$ if $|t-s| < \delta_2$. As a conclusion, we have found $\delta_2$ (depending only on $\epsilon$, $r$ and the input functions $\psi, g, h$) such that if $|w - x| < \delta_2$ and $|t-s| < \delta_2$ then $|f(w, s) - \psi(x, t)| < \epsilon$.

	\end{proof}
	
	
	\begin{proof}[Proof of Lemma~\ref{lem:bdary values of X-harmonic observables}]
		It is clear that condition (i) implies condition (ii), and hence it suffices to prove the statement by assuming condition (ii).  For this, 
		note first that $\diff \gamma_t = \gamma_t g(X_t) \diff t$ and $\gamma_0 = 1$, so $\gamma_t \leq e^{C t}$ by Grönwall's lemma. Hence $\gamma_\tau - 1 
		\leq e^{C\tau} -1.$
		It follows from the trivial bound $\gamma_\tau - 1 > -1$ and dominated convergence theorem that $(\gamma_\tau - 1)$ is integrable under $\mathbb{E}_w$ and converges to $0$ in $L^1$ as $w \to x$.
		Similarly, using the fact that $\diff H_t = h(X_t) \gamma_t \diff t$, one obtains 
		\begin{align}
			\label{eq:bound on H}
			|H_t| \leq \Vert h \Vert_\infty \int_{s=0}^t e^{Cs} \diff s \leq \Vert h \Vert_\infty \tfrac{e^{Ct} - 1}{C}
		\end{align}
		that is integrable by assumption, and consequently $\mathbb{E}_w[|H_\tau|] \to 0$ as $w \to x$. To conclude the proof, we use again
		\begin{align*}
			|f(w) - \psi (x)| \leq \mathbb{E}_w [|\gamma_\tau - 1|\cdot | \psi(X_\tau)| +  | \psi(X_\tau)  -  \psi(x)| + |H_\tau |].
		\end{align*}
		Here the first and the third term vanish in the limit $w\to x$ by above considerations, and the middle term can be handled as in the time-dependent case. This concludes the whole proof.
	\end{proof}

\end{document}